\documentclass{article}

\usepackage[ruled,longend,linesnumbered]{algorithm2e}

\usepackage{geometry}
\newgeometry{vmargin={24mm}, hmargin={33mm,53mm}}

\usepackage{marginnote}
\usepackage{lipsum}
\usepackage{amsfonts}
\usepackage{graphicx}
\usepackage{epstopdf}
\usepackage{algorithmic}
\usepackage{stmaryrd}
\usepackage{multirow}
\usepackage{amsmath}
\usepackage{amssymb}
\usepackage{mathtools}
\usepackage{pdfpages}
\usepackage{bbm}
\usepackage{diagbox}
\usepackage{amsbsy}
\usepackage{xcolor}
\usepackage{capt-of}
\usepackage{soul}
\usepackage{todonotes}
\usepackage{esint}
\usepackage{tikz}
\usetikzlibrary{calc}
\usetikzlibrary{calc,patterns}
\usepackage{ifthen}
\usepackage{pgfplots} 
\usepackage[utf8]{inputenc}
\pgfplotsset{compat=1.18} 

\definecolor{olive}{HTML}{bcbd22}
\definecolor{blue}{HTML}{1f77b4}
\definecolor{green}{HTML}{2ca02c}
\definecolor{red}{HTML}{d62728}
\definecolor{purple}{HTML}{9467bd}
\definecolor{mplred}{RGB}{214,39,40}

\DeclareMathOperator*{\argmin}{arg\,min}
\newcommand{\norm}[1]{\left\lVert#1\right\rVert}

\newcommand{\x}{\mathbf{x}}
\newcommand{\dx}{\mathrm{d}}
\newcommand{\Vh}{\mathcal{V}_h}
\newcommand{\Yh}{\mathcal{Y}_h}

\newcommand{\R}{\mathbb{R}}
\newcommand{\T}{\mathbb{T}}

\newcommand{\D}{\mathbf{D}}

\newcommand{\Qh}{\mathcal{Q}_h}
\newcommand{\WWh}{\mathcal{W}_h}
\newcommand{\Ycal}{\mathcal{Y}}

\newcommand{\dxS}{\mathrm{d}\Pi(\mathbf{x},t)}
\newcommand{\dxSh}{\mathrm{d}\Pi_h(\mathbf{x},t)}
\newcommand{\gradsurf}{\nabla_{\Gamma_H}}

\newcommand{\Qcal}{\mathcal{Q}}
\newcommand{\Vcal}{\mathcal{V}}

\usepackage[normalem]{ulem}
\newcommand{\soutr}{\bgroup\markoverwith{\textcolor{red}{\rule[0.5ex]{2pt}{1.2pt}}}\ULon}

\ifpdf
  \DeclareGraphicsExtensions{.eps,.pdf,.png,.jpg}
\else
  \DeclareGraphicsExtensions{.eps}
\fi

\usepackage{scalerel}[2014/03/10]
\usepackage[usestackEOL]{stackengine}
\def\avint{\,\ThisStyle{\ensurestackMath{%
    \stackinset{c}{0\LMpt}{c}{0\LMpt}{\SavedStyle-}{\SavedStyle\phantom{\int}}}%
    \setbox0=\hbox{$\SavedStyle\int\,$}\kern-\wd0}\int}

\usepackage{amsthm}

\newtheorem{theorem}{Theorem}
\newtheorem{lemma}{Lemma}     
\newtheorem{proposition}{Proposition}  
\newtheorem{definition}{Definition} 

\newtheorem{remark}{Remark}

\theoremstyle{remark}

\theoremstyle{plain}

\usepackage{hyperref}
\usepackage{cleveref}
\crefname{hypothesis}{hypothesis}{hypotheses}
\crefname{algocf}{algorithm}{algorithms}

\title{Learned Finite Element-based Regularization of the Inverse Problem in Electrocardiographic Imaging}

\author{Manuel Haas\footnotemark[2]
\and Thomas Grandits\footnotemark[3]
\and Thomas Pinetz\footnotemark[4]
\and Thomas Beiert\footnotemark[5]
\and Simone Pezzuto\footnotemark[6]
\and Alexander Effland\footnotemark[2]}

\begin{document}
	
\maketitle

\renewcommand{\thefootnote}{\fnsymbol{footnote}}

\footnotetext[0]{\url{https://github.com/mnlhaas/ECGI_FEM}}
\footnotetext[1]{This work was supported by the Deutsche Forschungsgemeinschaft (DFG, German
Research Foundation), EXC2151-390873048, EXC-2047/1-390685813, and -- CRC 1720 -- 539309657. The authors gratefully acknowledge access to the Marvin cluster of the University of Bonn. SP and TG are also supported by the SNSF project ``CardioTwin'' (no.~214817). SP acknowledges support from the European Research Council (ERC) under the European Union's Horizon Europe research and innovation programme (Grant agreement No. 101230597), the PRIN2022 project (MUR, Italy, 2023--2025, no.~P2022N5ZNP), and the CSCS Swiss National Supercomputing Centre (project no.~p157). SP is a member of INdAM-GNCS.}
\footnotetext[2]{Institute for Applied Mathematics, University of Bonn, Germany (\href{mailto:haas@iam.uni-bonn.de}{\nolinkurl{haas@iam.uni-bonn.de}}).}
\footnotetext[3]{Department of Mathematics and Scientific Computing, University of Graz, Austria}
\footnotetext[4]{ Institute of Artificial Intelligence, Medical University of Vienna,
Austria}
\footnotetext[5]{ Heart Center Bonn, Department of Internal Medicine II, University Hospital Bonn, Germany}
\footnotetext[6]{ Department of Mathematics, University of Trento, Italy}

\renewcommand{\thefootnote}{\arabic{footnote}}

\begin{abstract}
Electrocardiographic imaging (ECGI) seeks to reconstruct cardiac electrical activity from body-surface potentials noninvasively. 
However, the associated inverse problem is severely ill-posed and requires robust regularization. 
While classical approaches primarily employ spatial smoothing, the temporal structure of cardiac dynamics remains underexploited despite its physiological relevance. 
We introduce a space–time regularization framework that couples spatial regularization with a learned temporal Fields-of-Experts (FoE) prior to capture complex spatiotemporal activation patterns.
We derive a finite element discretization on unstructured cardiac surface meshes, prove Mosco-convergence, and develop a scalable optimization algorithm capable of handling the FoE term. 
Numerical experiments on synthetic epicardial data demonstrate improved denoising and inverse reconstructions compared to handcrafted spatiotemporal methods, yielding solutions that are both robust to noise and physiologically plausible.
\end{abstract}

\par

\par

% \MH{TODO}
% \ALEX{TODO}
% \SP{TODO}
% \TP{TODO}
% \TG{TODO}

\section{Introduction}
The heart is an electromechanical pump essential to sustaining life. Abnormal electrical activation may lead to life-threatening arrhythmias, which must be promptly diagnosed to prevent further damage. 
In electrocardiography, electrical potentials recorded on the chest surface are linked to the underlying cardiac electrical activity and can therefore provide insight into arrhythmic mechanisms and guide therapeutic decisions~\cite{Pe22}. 
From a mathematical perspective, reconstructing cardiac potentials from body-surface measurements leads to the inverse problem of electrocardiography~\cite{Cl18, Fr14}. This inverse problem is severely ill-posed, and its difficulty is further exacerbated by noise and measurement artifacts arising from imperfect electrode contact and anatomical constraints~\cite{Be07}. Consequently, the choice of regularization is critical for obtaining physiologically meaningful reconstructions.

A broad range of regularization techniques has been investigated to stabilize the reconstructions~\cite{Ka18}, ranging from classical Tikhonov methods~\cite{Ti77} to total variation~\cite{Ru92} approaches. 
While spatial regularization has been widely studied, the temporal dimension in medical imaging remains underutilized, despite its importance for capturing complex spatiotemporal activation patterns, such as the wave-like propagation observed in cardiac electrical potentials~\cite{Ha25, Ka26, On09}. 
Deep learning techniques demonstrate good accuracy in inverse problem reconstruction~\cite{Te22}, and optimization-based approaches in machine learning frameworks have been explored for patient-specific cardiac modeling~\cite{Gr24}, but provide limited interpretability.

Recently, machine-learned priors such as Fields-of-Experts (FoE)~\cite{Du25, Go24, Ro05} and variational deep regularizers~\cite{Ko22} have demonstrated strong performance in modeling higher-order dependencies in image reconstruction tasks and retain interpretability, yet their integration with PDE-constrained inverse problems in finite element frameworks remains largely unexplored.
In this work, we propose a novel class of learned finite element regularizers that combine spatial total variation–type terms with a data-driven temporal FoE prior. 
Both the cardiac inverse problem and the proposed regularization are discretized within a finite element framework~\cite{Br08,Ci78,Se05,Wa10}.

Beyond modeling, our main theoretical contribution is the first Mosco convergence analysis for inverse problems regularized by learned FoE-type energies in finite element spaces, establishing Mosco convergence and convergence of minimizers for a convex subclass of the proposed regularizers~\cite{Ef20, Mo69, Na24}.  
Numerical experiments demonstrate that the proposed learning–analysis framework produces reconstructions that are both more accurate and more physiologically plausible than those obtained by classical non-learning-based methods.
The paper is structured as follows: we first introduce a variational model with learned spatiotemporal regularization, then present its finite element discretization on unstructured cardiac meshes, establish Mosco convergence and convergence of minimizers, develop a scalable convergent optimization algorithm, and finally validate the approach with numerical experiments on synthetic 2D and 3D data.

\section{Methods}
We begin by introducing the forward and inverse problems in electrocardiographic imaging (ECGI), along with the FoE-based regularization.

\subsection{Forward Problem}
The forward problem consists of computing body-surface potentials resulting from a given epicardial potential, obtained by transferring the epicardial potential through the torso volume conductor via a linear forward operator.
Let $\Omega\subset \mathbb{R}^d$ with $d\in\{2,3\}$ denote the torso domain with outer boundary $\Gamma = \partial \Omega$. 
The body--surface electrodes are represented by a subset $\Sigma \subset \Gamma$, consisting of a disjoint union of $N_\Sigma$ open and bounded patches $\Sigma = \bigcup_{i=1}^{N_\Sigma} \Sigma_i$ for each electrode, where measurements $z$ are available.
The epicardial surface of the heart is denoted by $\Gamma_H = \partial \Omega_H$, where $\Omega_H \subset \Omega$ is the heart domain. 
We define the torso domain excluding the heart as $\Omega_0 = \Omega \setminus \overline{\Omega}_H$.
Throughout, we assume that the domain $\Omega_0$ possesses a $C^2$-boundary, and that both $\Gamma_H$ and $\Gamma$ are closed and compact disjunct boundary manifolds.
For a time interval $\T\coloneqq(0,\widetilde{t})$ with $\widetilde{t}>0$ and the epicardium potential function $u\in L^2(\T;H^{1/2}(\Gamma_H))$, the well-posed elliptic forward problem is defined for almost every $t\in \T$
\begin{equation}
\label{eq:E_forward}
\begin{cases}
\begin{aligned}
-\operatorname{div}\left(\sigma(\x)\nabla_{\x}v(\x,t)\right) &= 0,  &&\mathbf{x}\in \Omega_0, \\
\sigma(\x)\nabla_{\x}v(\x,t)\cdot\mathbf{n} &= 0,  &&\mathbf{x}\in \Gamma,  \\
v(\x,t) &= u(\x,t),  &&\mathbf{x}\in \Gamma_H,
\end{aligned}
\end{cases}
\end{equation}
for $\mathbf{n}$ the outward unit normal vector and $\sigma\in L^\infty(\Omega,\mathbb{R}_+^{d\times d})$ the torso conductivity tensor or bulk conductivity.
Here, we assume the space-dependent matrix $\sigma$ to be symmetric satisfying the ellipticity condition~$\zeta^{-1}|\mathbf{y}|^2\leq \sigma(\x)\mathbf{y}\cdot \mathbf{y}\leq \zeta|\mathbf{y}|^2$ for some $\zeta \geq 1$ and all $\mathbf{y} \in \R^d$.
The existence, uniqueness, and regularity of a solution follow~\cite{Ev10}.
\begin{theorem}
There exists a unique solution $v\in L^2(\T;H^1(\Omega_0))$ of \eqref{eq:E_forward} for $u\in L^2(\T;H^{1/2}(\Gamma_H))$ in a weak sense.
\begin{proof}
Let $E:H^{1/2}(\Gamma_H)\to H^1(\Omega_0)$ be a bounded linear right inverse of the trace operator $\mathrm{Tr}_{\Gamma_H}$ and set $g(\cdot,t)=E[u(\cdot,t)]$. 
Then $g|_{\Gamma_H}=u(\cdot,t)$ and
\[
\|g(\cdot,t)\|_{H^1(\Omega_0)}
\le C_E \|u(\cdot,t)\|_{H^{1/2}(\Gamma_H)}.
\]
Moreover, since $E$ is bounded and linear,
$g\in L^2(\T;H^1(\Omega_0))$.
Writing $v(\cdot,t)=g(\cdot,t)+w(\cdot,t)$ with $w(\cdot,t)\in H_{0,\Gamma_H}^1(\Omega_0)
\coloneqq
\{y\in H^1(\Omega_0):\mathrm{Tr}_{\Gamma_H}y=0\}$ reduces \eqref{eq:E_forward} to homogeneous Dirichlet boundary conditions on $\Gamma_H$.  
Defining
$a(\phi,\varphi):=\int_{\Omega_0}\sigma\nabla_\x\phi\cdot
\nabla_\x\varphi\,\mathrm d\x$
with $\phi,\varphi\in H_{0,\Gamma_H}^1(\Omega_0)$, which is
continuous and, by the ellipticity of $\sigma$ and the Poincar\'e
inequality on $H_{0,\Gamma_H}^1(\Omega_0)$, coercive, the
Lax--Milgram theorem gives for a.e.\ $t\in\T$ a unique $w(\cdot,t)\in H_{0,\Gamma_H}^1(\Omega_0)$ solving
\begin{equation}\label{eq:spatial_lifting}
a(w(\cdot,t),\varphi) = -a(g(\cdot,t),\varphi) \quad \forall \varphi\in H_{0,\Gamma_H}^1(\Omega_0),
\end{equation}
with $\|w(\cdot,t)\|_{H^1(\Omega_0)} \leq C\|g(\cdot,t)\|_{H^1(\Omega_0)}$. 
Hence $v(\cdot,t)$ satisfies the weak formulation of \eqref{eq:E_forward} and $\|v(\cdot,t)\|_{H^1(\Omega_0)} \le C\|u(\cdot,t)\|_{H^{1/2}(\Gamma_H)}$. 
Since the solution operator associated with the Lax--Milgram problem
is bounded and linear, $w$, and hence $v$, is strongly measurable.
Integration over $\T$ therefore yields $v\in L^2(\T;H^1(\Omega_0))$, and uniqueness follows from coercivity.
\end{proof}
\end{theorem}
Let $v_{u}(\cdot,t)$ be the weak solution associated with the boundary data $u(\cdot,t)$ for a.e. $t\in \T$.
We define the operator
\[
A_0 : H^{1/2}(\Gamma_H) \to H^{1/2}(\Gamma),
\quad \text{with} \quad
A_0[u(\cdot,t)] \coloneqq v_u(\cdot,t)|_{\Gamma}
\] 
as the solution operator concatenated with the trace operator for the torso boundary.
Therefore, the forward operator is a bounded linear operator, and since $u(\cdot,t):\T\to H^{1/2}(\Gamma_H)$ is strongly measurable and $A_0$ is bounded and linear, the composition $A_0[u(\cdot,t)]$ is strongly measurable, ensuring boundedness by
\begin{equation*}
    \norm{A_0[u(\cdot, t)]}_{H^{1/2}(\Gamma)}\leq C\norm{v(\cdot, t)}_{H^1(\Omega_0)}\leq \widetilde{C}\norm{u(\cdot, t)}_{H^{1/2}(\Gamma_H)}.
\end{equation*}
By integrating over time, we define the linear forward operator
\[
A:L^2(\T; H^{1/2}(\Gamma_H))\to L^2(\T; H^{1/2}(\Gamma)), \quad u\mapsto v_u|_\Gamma.
\]
Note that the inverse of the forward operator is, in general, unbounded.

\subsection{Inverse Problem}
For the inverse problem in electrocardiography, we compare to a time series of measurements $z\in (L^2(\T))^{N_\Sigma}$ for each electrode $i=1,\ldots,N_\Sigma$ defined by the average integral $z_i(t)=\avint_{\Sigma_i} \widetilde{z}(\x,t) \dx S^\Gamma(\x)$ of a potential function $\widetilde{z}\in L^2(\Sigma\times \T)$ with the $d-1$-dimensional surface measure $S^\Gamma$ of $\Gamma$.
Since the measurements are defined electrode-wise as scalars, we minimize the cost functional 
\begin{equation*}
\argmin_{u\in L^2(\T;H^{1/2}(\Gamma_H))} 
\left\{
G(u,z)\coloneqq\frac{1}{2 N_\Sigma}\sum_{i=1}^{N_\Sigma}\int_\T\left(\avint_{\Sigma_i}A[u](\x,t)\dx  S^\Gamma(\x) -z_i(t)\right)^2 \dx t
\right\}.
\end{equation*}
Due to the ill-posedness of the problem, we follow the standard approach in inverse problem theory and introduce a regularization term. 
The inverse problem is formulated as the minimization of an energy functional 
consisting of a data fidelity term and a regularization term. 
For a given measurement $z\in (L^2(\T))^{N_\Sigma}$, we seek an epicardial potential $u$ that minimizes
\begin{equation}\label{eq:energy}
    \argmin_{u \in \mathcal{V}} 
    \left\{\mathcal{J}(u, z) \coloneqq G(u, z) + R_\theta(u)\right\},
\end{equation}
over the space $\mathcal{V} \coloneqq L^2(\T; H^1(\Gamma_H))$.
The choice of $\Vcal$ as space of solutions follows from the continous embedding $H^1(\Gamma_H) \hookrightarrow H^{1/2}(\Gamma_H)$ on compact $C^2$ manifolds, hence  $\mathcal{V} \subset L^2(\T; H^{1/2}(\Gamma_H))$.
The norm of the Bochner space $\Vcal$ is defined as
\begin{equation*}
\norm{u}_\Vcal = \left(\int_\T \norm{u(\cdot, t)}_{H^1(\Gamma_H)}^2\dx t\right)^{\frac{1}{2}}=\left(\int_{\Gamma_H\times \T} (u(\x,t))^2 + |\gradsurf u(\x,t)|^2 \dxS\right)^{\frac{1}{2}}
\end{equation*}
with surface gradient $\gradsurf$ for any $v:\Gamma_H \to \mathbb{R}$ and $\dxS\coloneqq \dx S(\x)\otimes \dx \lambda(t)$ with $S$ the $(d-1)$-dimensional surface measure of $\Gamma_H$ and $\lambda$ the $1$-dimensional Lebesgue measure.
Moreover, we define the space $\Ycal\coloneqq L^2(\Gamma_H\times \T)$ to allow for more convenient notation of norms.

\subsection{Spatiotemporal Regularizer}
Standard regularization techniques such as Tikhonov~\cite{Ti77} and total variation~\cite{Ru92} are general-purpose approaches that effectively mitigate measurement noise. Here, we aim to design a problem-specific regularizer that embeds data-driven knowledge of the underlying cardiac dynamics while remaining mathematically well-posed and interpretable.

To this end, we adopt the FoE framework, originally developed for imaging, in which, instead of prescribing a fixed penalty, one learns a collection of experts~\cite{Du25,Go24,Ro05}. Each expert \( i = 1,\dots,N_C \) is defined by a bounded linear operator \( L_i \) acting on the epicardial potential and an associated nonlinear potential function \( \phi_i \). The operator response \( L_i[u] \) is penalized pointwise in space and time via \( \phi_i \). Overall, the regularizer in a continuous spatiotemporal formulation reads as
\begin{equation}\label{eq:FoE}
    R_\theta(u) 
    = \sum_{i=1}^{N_C} 
    \int_{\Gamma_H \times \T} 
    \phi_i\left(L_i[u](\x,t)\right)\dxS.
\end{equation}
In canonical FoE, the operators $L_i$ are chosen as convolutions with learned kernels, mirroring convolutional layers in neural networks.
However, spatiotemporal convolutions on unstructured finite element grids are computationally expensive due to their irregular structure, whereas temporal convolutions on a typically uniform time grid are straightforward to compute.
We therefore propose a multivariate operator that separates spatial differentiation from temporal filtering while still allowing joint, learned coupling of their responses. 
In our formulation, the operators $L_i$ consist of three components: a zero-order term $\epsilon_\theta u$ ensuring stability, the surface gradient $\gradsurf u$ imposing spatial regularity, and a temporally filtered component $K_i[u]$ such that
\begin{equation}\label{eq:lin_op}
    L_i[u](\x,t)\in\mathbb{R}\times\mathbb{R}^{d}\times\mathbb{R} \quad  \text{with} \quad L_i[u]\coloneqq (\epsilon_\theta u, \gradsurf u, K_i[u]),
\end{equation}
where $\epsilon_\theta>0$ is a learned scalar and $K_i$ is realised as a bounded temporal cross correlation with a 1D kernel $k_i\in L^1(\widetilde{\T})$ with $\widetilde{\T}=(-\omega/2,\omega/2)$ for $\omega>0$ applied to the zero-extended signal $\widetilde{u}$
\begin{equation}\label{eq:extended_u}
K_i[u](\x,t)\coloneqq(k_i \ast_\T \widetilde{u})(\x,t) 
\coloneqq  
\int_{\widetilde{\T}} 
k_i(\tau) \widetilde{u}(\x,\tau+t) \mathrm{d}\tau, \quad t\in\T,
\end{equation} 
with $\widetilde{u}=u$ on $\T$ and $\widetilde{u}=0$ otherwise.
We adapt the multivariate setting introduced in~\cite{Du25} such that the corresponding potentials $\phi_i : \mathbb{R}^{d+2} \to \mathbb{R}$ penalize the joint response of the components in~\eqref{eq:lin_op}.
The operator $L_i$ extracts local spatial and temporal features of the epicardial potential. 
In particular, $L_i[u](\x,t)$ jointly encodes the local signal amplitude, spatial variation, and temporally filtered activity. 
The potential $\phi_i$ assigns low energy to feature vectors close to the origin, corresponding to weak filter responses favored by the learned prior, while increasingly penalizing large deviations.
The role of the potential $\phi_i$ is to assign low energy to such weak responses, which are regarded as consistent with the learned prior, while increasingly penalizing stronger responses.
Hence,~\eqref{eq:FoE} becomes
\begin{equation}\label{eq:FoE_multivariate}
    R_\theta(u) 
    = \lambda_\theta \sum_{i=1}^{N_C} 
    \int_{\Gamma_H \times \T} 
    \phi_i\left(\epsilon_\theta u(\x,t), \gradsurf u(\x,t), (k_i \ast_\T u)(\x,t)\right)\dxS,
\end{equation}
with weighting parameter $\lambda_\theta > 0$.
The filters $k_i$ and parameters of $\phi_i$ are learned from data, jointly with $\epsilon_\theta$ and $\lambda_\theta$; the spatial operator $\gradsurf$ is fixed by the geometry.

Following~\cite{Du25}, we construct multivariate potentials to generalize the well performing one-dimensional weakly convex ridge regularizer (WCRR)~\cite{Go24} approach to multiple dimensions. 
The learned WCRR potentials can be represented accurately as differences of scaled Moreau envelopes of the absolute value.
Motivated by this observation, we chose the  Moreau envelope~\cite{Mo65} of the $\ell^\infty$-norm to generalize the absolute value from the scalar case to the multivariate setting.
Besides preserving this connection to the scalar case, the transformed $\ell^\infty$-norms are sufficiently expressive to approximate a broad class of convex norms through suitable linear transformations~\cite{Un26}, while their Moreau envelopes admit efficient evaluation and possess Lipschitz-continuous gradients, enabling efficient gradient-based optimization.
Given a proper, lower semi-continuous, convex function $f:\mathbb{R}^{d+2} \to \mathbb{R}$ and a parameter $\mu>0$, its Moreau envelope is defined as $ M_\mu f(\mathbf{x}) = \inf_{\mathbf{z}\in\mathbb{R}^{d+2}} \left\{ f(\mathbf{z}) + \frac{1}{2\mu} \|\mathbf{z}-\mathbf{x}\|_{2}^2 \right\}$.
The Moreau envelope can be viewed as a smooth approximation of $f$ that preserves key properties of the original function while ensuring differentiability and well-behaved gradients. In particular, its gradient is directly related to the proximal operator of $f$, $\nabla M_\mu f = \frac{1}{\mu} \left( \mathbf{I}_{d+2} - \mathrm{prox}_{\mu f} \right)$, which is $\frac{1}{\mu}$-Lipschitz continuous and computationally efficient to evaluate for many choices of $f$. These properties make the Moreau envelope suitable for designing potentials in optimization-based inverse problems, where stable and efficient gradient computation is crucial. 
Building on this idea, we define our multivariate potential functions as differences of Moreau envelopes of the $\ell^\infty$-norm. The construction follows the multivariate FoE framework of~\cite{Du25}. 
The Moreau envelope of the $\ell^\infty$-norm is computed by
\begin{equation*}
    M_\mu \norm{\mathbf{y}}_\infty = \left\| \mathbf{y} - \mu\,\mathrm{Proj}_{B_{\ell^1}}\left(\mathbf{y}/\mu\right) \right\|_{\infty}
+ \frac{\mu}{2}\left\| \mathrm{Proj}_{B_{\ell^1}}\left(\mathbf{y}/\mu\right) \right\|_2^2
\end{equation*}
for $\mathrm{prox}_{\iota_{B_{\ell^1}}}(\mathbf{y}) = \mathrm{Proj}_{B_{\ell^1}}(\mathbf{y})  = \argmin_{\mathbf{z}\in B_{\ell^1}}\{\norm{\mathbf{z}-\mathbf{y}}_2^2\},$
the projection onto the $\ell^1$-ball in $\mathbb{R}^{d+2}$ for indicator function of the $\ell^1$-ball $\iota_{B_{\ell^1}}$.
Here, we used Moreau's identity to compute the minimizer $\mathbf{z}^*$ of the Moreau envelope
\begin{equation*}
\mathbf{z}^*=\mathrm{prox}_{\mu\norm{\cdot}_{\ell^\infty}}(\mathbf{y})= \mathbf{y} - \mu \mathrm{Proj}_{B_{\ell^1}}(\mathbf{y}/\mu).
\end{equation*}
To ensure coercivity in the inverse problem, we augment the Moreau envelope with a small quadratic term analogous to zero-order Tikhonov regularization
\begin{equation*}
\omega_\mu^{d+2}(\mathbf{y})
= M_\mu \norm{\mathbf{y}}_\infty
+ \frac{\epsilon_\omega}{2\mu} \left\| \mathbf{y}\right\|_2^2.
\end{equation*} 
For constructing a convex regularizer of the form~\eqref{eq:FoE_multivariate}, the potential functions are chosen as
\begin{equation}\label{eq:convex_potentials}
\phi_i=\mu_i\omega_{\mu_i}^{d+2} \quad \text{for} \quad i=1,\ldots,N_C.
\end{equation}
However, following the multivariate FoE framework, we construct the final potential as a difference of two Moreau-envelope-based terms. 
This yields a differentiable weakly convex (and generally nonconvex) potential under the conditions of~\cite[Theorem 1]{Du25}, with increased expressive power that allows the model to capture more complex interactions between operator responses while balancing convexity and flexibility
\begin{equation}\label{eq:potential_function}
    \phi_i(\mathbf{y}) = \mu_i \omega_{\mu_i}^{d+2}(\mathbf{y}) 
    - \mu_i \omega_{\eta_i \mu_i}^{d+2}(\mathbf{Q}_i \mathbf{y})\quad \text{for} \quad i=1,\ldots,N_C,
\end{equation}
with learnable matrices $\mathbf{Q}_i \in \mathbb{R}^{(d+2)\times(d+2)}$ and scalars $\eta_i$.
Since $\omega_\mu^{d+2}$ is the Moreau envelope of the $\ell^\infty$-norm with a quadratic term, the potential favors operator responses whose components remain jointly small while assigning increasing penalties to responses with large magnitude. 
The learnable transformation $\mathbf{Q}_i$ couples the spatial and temporal components, enabling the model to learn correlations between them rather than treating each component independently. The appearance of the projection onto the $\ell^1$-ball is not an additional modeling choice but follows from the closed-form expression of the proximal operator of the $\ell^\infty$-norm.
The derivative of $\omega_\mu^{d+2}$,
$(\omega_\mu^{d+2})'(\mathbf{y}) = \mathrm{Proj}_{B_{\ell^1}}(\mathbf{y}/\mu) + \epsilon_\omega \mathbf{y}/\mu$,
is computationally efficient and Lipschitz continuous with Lipschitz constant $(1+\epsilon_\omega)/\mu$, ensuring stability in gradient-based optimization.

According to \cite[Theorem 1]{Du25}, the conditions $\|\mathbf{Q}_i\|_\infty \le 1$ and $\eta_i > \|\mathbf{Q}_i\|_2^2$ guarantee that the potential $\phi_i$ is nonnegative and attains a unique global minimum at the origin, which carries over naturally to our framework and ensures that the learned regularizer is bounded from below while favoring weak operator responses. 
We also note that the regularization functional $R_\theta$ is differentiable, allowing for efficient gradient-based optimization.
While the subsequent optimization framework only requires smoothness and boundedness properties of the resulting regularizer, the proposed Moreau-envelope-based parameterization provides a structured and learnable family of potentials that satisfies these requirements by construction while allowing the model to capture multivariate interactions between spatiotemporal features.

\section{Finite Element Discretization}
In this section, we introduce the discrete domains, function spaces, and operators used throughout this work, constructed via the finite element method (FEM).

\subsection{Space-Time Discretization}
We define $\mathcal{T}_h$ with largest element diameter $h>0$ as the collection of surface elements discretizing $\Gamma_H$ and, with a slight misuse of notation, the shape-regular adaptive conforming surface mesh approximating $\Gamma_H$.
Furthermore, we assume that the elements in $\mathcal{T}_h$ are restricted to vertices lying on $\Gamma_H$.
Analogously, we discretize $\Gamma$ and consistent as subsets $\Sigma_i$ by 
$\mathcal{T}_h^{\Gamma}$ and $\mathcal{T}_h^{\Sigma_i}$ for $i = 1, \ldots, N_\Sigma$.
The torso domain $\Omega_0$ is  discretized by a mesh
$\mathcal{U}_h$, which is defined as a shape-regular volumetric discretization satisfying $\mathcal{T}_h\cup \mathcal{T}_h^{\Gamma}=\partial \mathcal{U}_h$ either consisting of triangles ($d=2)$ or tetrahedra ($d=3$).

The temporal interval $\T$ is discretized as a uniform temporal grid $\mathcal{S}_h$. 
In particular, we assume that temporal measurements are available at $N_\T+1$ equidistant time points $t_s=s\delta$ for $s=0,\ldots,N_\T$
which define $N_\T$ time intervals of size $\delta > 0$, with $\delta = O(h)$.

\subsection{Function Spaces}
We approximate the continuous space-time function space $\mathcal{V}$ by the discrete space $\mathcal{V}_h$ using $P_1$ finite elements in space on the adaptive surface mesh $\mathcal{T}_h$ and in time on the temporal grid $\mathcal{S}_h$. 
The spatial mesh is identical at each time step and consists of affine elements of dimension $d-1$. 
Let $\pmb{\Xi} \coloneqq \{\pmb{\xi}_i\}_{i=1}^{N_{\mathcal{V}}}\subset \Gamma_H$ denote the set of nodes associated with 
$\mathcal{T}_h$. 
The tensor product of a spatial element $J \in \mathcal{T}_h$ and a temporal interval $Z \in \mathcal{S}_h$ 
defines a space-time element $J \otimes Z \in \mathcal{T}_h \otimes \mathcal{S}_h$.
The resulting discrete space-time function space is
\begin{equation*}
\Vh \coloneqq \Big\{ v \in L^2(\T;H^1(\mathcal{T}_h)) \colon v|_{J \otimes Z} \in P_1(J) \otimes P_1(Z),  \forall J \in \mathcal{T}_h, Z \in \mathcal{S}_h \Big\},
\end{equation*}
with spatial and temporal basis functions $\{\varphi_i\}_{i=1}^{N_{\Vcal}}$ and $\{\rho_s\}_{s=0}^{N_\T}$, respectively. Any discrete function $u_h \in \Vh$ can be expressed as
\[
u_h(\x,t) = \sum_{s=0}^{N_\T} \sum_{i=1}^{N_{\Vcal}} \mathbf{u}_{i,s} \, \varphi_i(\x) \rho_s(t),
\]
where $\mathbf{u}_{i,s}$ is the value at node $(\pmb{\xi}_i, t_s)$ with $\mathbf{u}=(\mathbf{u}_{1,0},\mathbf{u}_{2,0},\ldots,\mathbf{u}_{N_\Vcal, N_\T})\in \mathbb{R}^{N_\Vcal (N_\T+1)}$.
Next, we examine the function spaces induced by the linear operators $L_i$. 
Since $P_1$ functions are affine on each mesh element $T \in \mathcal{T}_h$, 
their spatial gradients are element-wise constant and belong to the broken $L^2$ space 
\[
P_0(\mathcal{T}_h) := \{ v \in L^2(\mathcal{T}_h) : v|_T \text{ is constant for all } T \in \mathcal{T}_h \}.
\]
This defines a discrete gradient space $\Qh$ composed of vector-valued functions $p_h = (p_{h_1}, \ldots, p_{h_d})^\top $ with $ p_{h_k} \in L^2(\mathcal{T}_h \times \T)$ for $k=1,\ldots d$, that are piecewise constant in space per element and affine in time:
\begin{equation*}
\Qh \coloneqq \Big\{ p \in \left(L^2(\mathcal{T}_h\times \T)\right)^d \colon p_k|_{J \otimes Z} \in P_0(J) \otimes P_1(Z), \forall J \in \mathcal{T}_h, Z \in \mathcal{S}_h \Big\}.
\end{equation*}
Here, $\{\vartheta_l\}_{l=1}^{N_{\Qcal}}$ denotes the basis of $P_0(\mathcal{T}_h)$ with $N_{\Qcal} \in \mathbb{N}$.
The temporal cross-correlation in the regularizer function acts on functions $u_h \in \Vh$ as
$k_h\ast_\T \cdot \colon \Vh \to \WWh$,
where the cross-correlation of piecewise affine functions yields a piecewise cubic one
\begin{equation*}
\WWh \coloneqq \Big\{ w \in L^2(\T;H^1(\mathcal{T}_h)) \colon w|_{J \otimes Z} \in P_1(J) \otimes P_3(Z), \forall J \in \mathcal{T}_h, Z \in \mathcal{S}_h \Big\}.
\end{equation*}
We approximate time-dependent functions using the standard piecewise linear hat functions $\rho_i$ on the uniform temporal grid $t_0,\dots,t_{N_\T}$ with step size $\delta$, defined by $\rho_i(t)=
\max\{ 0, 1 - |t - t_i|/\delta \}$ for $i = 0,\dots,N_\T$.
These functions naturally lead to the temporal mass matrix
\[
\widetilde{\mathbf{D}}_{ij} = \int_\T \rho_i(t) \rho_j(t)\, dt
= \frac{\delta}{6}
\begin{cases}
4, & i=j, 1\le i \leq N_\T-1,\\
2, & i=j, i\in\{0,N_\T\},\\
1, & |i-j|=1,\\
0, & \text{otherwise}.
\end{cases}
\]
In space, we employ piecewise affine basis functions $\varphi_i$ on the mesh $\mathcal{T}_h$, leading to the spatial mass matrix
\[
\widetilde{\mathbf{M}}_{ij} = \int_{\mathcal{T}_h} \varphi_i(\x) \varphi_j(\x)\, d\x, \quad 1\le i,j \le N_\mathcal{V}.
\]
To handle functions that depend on both space and time, we combine the temporal and spatial matrices using tensor products $\mathbf{D} = \widetilde{\mathbf{D}} \otimes \mathbf{I}_{N_\mathcal{V}}$ and $\mathbf{M} = \mathbf{I}_{N_\T+1} \otimes \widetilde{\mathbf{M}}$.
This construction ensures that the resulting matrices correctly account for all degrees of freedom of the discrete solution $u_h \in \mathcal{V}_h$, linking each spatial node with every time step.
Finally, for vector-valued functions with $d$ components, we define the standard $L^2$ inner product on the space--time domain $\mathcal{T}_h \times \T$ by
\[
(v_h, w_h)_{\mathcal{Y}_h^d} = \sum_{i=1}^{d} \int_{\mathcal{T}_h \times \T} v_{h,i}(\x,t)\, w_{h,i}(\x,t)\, d(\x,t),
\]
where $\mathcal{Y}_h := L^2(\mathcal{T}_h \times \T)$. 
This inner product provides the natural framework for measuring errors and formulating the fully discrete variational problem.

\subsection{Forward Operator}
The discretization of the forward operator $A$ in piecewise affine finite elements follows~\cite{Wa10}.
For almost all $t \in \T$, we consider the discrete version of~\eqref{eq:spatial_lifting}
\begin{equation}\label{eq:discrete_problem}
    a_h(v_h(\cdot,t),\varphi_h)=\int_{\mathcal{U}_h} \sigma_h(\x) \nabla_\x v_h(\x, t)\cdot\nabla_\x \varphi_h(\x) \dx \x=0 \quad \forall \varphi_h\in P_1(\mathcal{U}_h),\varphi_h|_{\mathcal{T}_h}=0
\end{equation}
with boundary conditions on the epicardium given by $v_h|_{\mathcal{T}_h} = u_h$ for some $u_h \in \Vh$.
Equivalently, we set $v_h = w_h + g_h$, where $w_h|_{\mathcal{T}_h} = 0$ and $g_h|_{\mathcal{T}_h} = u_h$, leading to
\begin{equation}\label{eq:discrete_problem_2}
a_h(w_h(\cdot,t),\varphi_h) = - a_h(g_h(\cdot,t),\varphi_h) \quad \forall \varphi_h \in P_1(\mathcal{U}_h), \ \varphi_h|_{\mathcal{T}_h} = 0.
\end{equation}
By coercivity and continuity of $a_h$, for fixed boundary data $u_h$,~\eqref{eq:discrete_problem} admits a unique solution $v_h$. 
 Over all timesteps, the discrete forward operator $A_h$ is then defined by restricting $v_h$ to the torso boundary, i.e. $A_h[u_h]=v_h|_{\mathcal{T}_h^\Gamma}$. 
In practice, this corresponds to solving a linear system for the piecewise affine finite element coefficients of $v_h$, yielding the mapping $\mathbf{v}_\mathbf{u} = \mathbf{A} \mathbf{u}$ from epicardium potentials to the torso boundary.

\subsection{Linear Operators in the Regularizer}
The regularization functional involves linear operators of the form $L[u] = (\epsilon_\theta u, \gradsurf u, K[u])$.
For the spatial gradient, we exploit the fact that piecewise affine functions on a manifold mesh admit an exact gradient representation.
The restriction $u_h|_J$ for $J\in \mathcal{T}_h$ with nodes $(\pmb{\xi}_i)_{i=1}^d$ is affine, and consequently, the discrete gradient is defined elementwise as
\[
\gradsurf^h u_h|_J =(\mathbf{I}_d - \mathbf{n}_J \otimes \mathbf{n}_J) \sum_{i=1}^d u_h(\pmb{\xi}_i) \nabla \varphi_i|_J,
\]
where $\{\varphi_i\}_{i=1}^d$ are the local nodal basis functions and $\mathbf{n}_J$ denotes the unit normal of $J$.
Since $u_h|_J$ is affine, $\nabla_{\Gamma_H}^h u_h|_J$ is constant 
on each element $J$ and hence belongs to $P_0(\mathcal{T}_h)^d$.
We denote by $\pmb{\nabla}_{\Gamma_H}^h$ the matrix representation of $\gradsurf^h$ acting on the nodal values of piecewise affine functions.
For the convergence analysis, let $k \in L^1(\widetilde{\T})$ and $k_h \in P_1(\widetilde{\mathcal{S}}_h)$ on a discretization $\widetilde{\mathcal{S}}_h$ using the standard $L^2$-projection. 
We assume that the resulting approximations satisfy $k_h \to k \quad \text{strongly in } L^1(\widetilde{\T})$.
The discrete kernel $k_h$, represented by its nodal vector $\mathbf{k}$, 
induces the temporal cross-correlation operator $K_h$. 
The kernel $k_h$ is defined on the extended interval 
$\widetilde{\T} = [t_{-N_w},t_{N_w}]$, with $N_w \in \mathbb{N}$, 
while $u_h$ is defined on $\T$.
The temporal cross-correlation $K_h[u_h](\x,t) = (k_h \ast_\T u_h)(\x,t)$ is discretized with the zero extended $\widetilde{u}$ from \eqref{eq:extended_u} as
\begin{align*}
K_h[u_h](\x,t) &= \int_{\widetilde{\T}} k_h(\tau) \widetilde{u}_h(\x, \tau + t) \dx\tau\\
&= \sum_{i=1}^{N_\Vcal} \sum_{l=0}^{N_\T} \sum_{j=0}^{2 N_w} \mathbf{u}_{i,l} \mathbf{k}_j \left(\int_{\widetilde{\T}} \rho_{j-N_w}(\tau) \mathbf{1}_\T(\tau + t)\rho_l(\tau + t) \dx\tau\right) \varphi_i(\x).
\end{align*}
Let $\mathbf{K}$ denote the matrix of $K_h$.
The resulting discretized operator is
\begin{equation}\label{eq:discrete_operator_1}
   L_h:\Vh \to \Vh\times \Qh\times \WWh,\quad L_h[u_h] \coloneq (\epsilon_\theta u_h, \gradsurf^h u_h, K_h[u_h]).
\end{equation}
For the nodal evaluation of the nonlinear Moreau-envelope terms, all components of the operator response must be represented on a common set of degrees of freedom.
We introduce an interpolation operator $\mathcal{P}_h: \Vh\times \Qh\times \WWh\to(\Vh)^{d+2}$ such that $\widetilde{L}_h=\mathcal{P}_h \circ L_h$.
Note that the evaluation of the different dimensions in the linear operator in the same finite element space is not necessary for the squared $\ell^2$-norm part of $\phi_i$.
We define a variant $\widetilde{\phi}_i$ of $\phi_i$ such that
\begin{equation}\label{eq:phi_tilde}
    \widetilde{\phi}_i(\mathbf{y},\widetilde{\mathbf{y}}) 
= \mu_i M_{\mu_i}\norm{\widetilde{\mathbf{y}}}_\infty - \mu_i M_{\eta_i\mu_i}\norm{\mathbf{Q}_i\widetilde{\mathbf{y}}}_\infty
+ \frac{\epsilon_\omega}{2} \left(\left\| \mathbf{y} \right\|_2^2-\frac{\left\|\mathbf{Q}_i\mathbf{y} \right\|_2^2}{\eta_i}\right),
\end{equation}
accounting for the difference in $L_h$ and $\widetilde{L}_h$.
We start by defining the temporal interpolation operator 
$\mathcal{P}_h^{\mathrm{temp}}$.
The cross-correlation of two piecewise affine functions is continuous 
and piecewise cubic in time.
To represent the resulting cross-correlation in the piecewise affine 
temporal finite element space, we evaluate it at the temporal nodes and 
define
\[
\mathcal{P}_h^{\mathrm{temp}} : C^0(\overline{\T}) 
\to P_1(\mathcal{S}_h),
\qquad
\mathcal{P}_h^{\mathrm{temp}}[v](t)
= \sum_{s=0}^{N_\T} v(t_s)\rho_s(t).
\]
Restricted to $P_3(\mathcal{S}_h)$, the interpolation operator is 
uniformly bounded in $L^2(\T)$, i.e.,
\[
\norm{\mathcal{P}_h^{\mathrm{temp}}v_h}_{L^2(\T)}
\leq C\norm{v_h}_{L^2(\T)}
\qquad
\text{for all }v_h\in P_3(\mathcal{S}_h),
\]
where $C>0$ is independent of $h$.
We further define a spatial interpolation operator 
$\mathcal{P}_h^{\text{sp}} \colon L^2(\mathcal{T}_h) \to P_1(\mathcal{T}_h)$
as the $L^2$-projection onto the subspace of piecewise affine spatial functions, allowing for computation of the piecewise affine projection of piecewise constant functions.
The aforementioned $L^2$-projection is defined by 
\begin{equation*}
    \mathcal{P}_h^{\text{sp}}[p] = \argmin_{q_h\in P_1(\mathcal{T}_h)} \left\{\norm{p-q_h}_{L^2(\mathcal{T}_h)}\right\}.
\end{equation*}
This orthogonal projection is computed by its optimality condition
\begin{equation}\label{eq:l2_proj}
    \int_{\mathcal{T}_h}\widetilde{p}_{h}(\x) \varphi(\x)\dx \x= \int_{\mathcal{T}_h}p_{h}(\x) \varphi(\x)\dx \x,\quad \forall \varphi \in P_1(\mathcal{T}_h)
\end{equation}
with $\widetilde{p}_{h}\coloneqq \mathcal{P}_h^{\text{sp}}[p_h]$.
By choosing $\varphi$ as the basis functions $(\varphi_i)_{i=1}^{N_\Vcal}$ and $p_h\in P_0(\mathcal{T}_h)$, we compute the nodal values $\widetilde{\mathbf{p}}$ of the projection by
\begin{equation*}
    \widetilde{\mathbf{M}}\widetilde{\mathbf{p}}=\mathbf{b} \quad \text{with} \quad \mathbf{b}_i = \int_{\mathcal{T}_h} p_h(\x)\varphi_i(\x)\dx\x=\sum_{j\colon \pmb{\xi}_i \in J_j}\mathbf{p}_j\int_{J_j} \varphi_i(\x)\dx\x=\sum_{j\colon \pmb{\xi}_i \in J_j}\mathbf{p}_j\frac{|J_j|}{d}
\end{equation*}
with $\varphi_i$ the basis function associated with node $\pmb{\xi}_i$ for $i=1,\ldots,N_\Vcal$.
We define the matrix representations of $\mathcal{P}_h^{\text{sp}}$ and $\mathcal{P}_h^{\text{temp}}$ as $\mathbf{P}^{\text{sp}}$ and $\mathbf{P}^{\text{temp}}$, respectively.
These interpolations allow us to define
\begin{equation}\label{eq:discrete_operator_2}
    \widetilde{L}_h\colon \Vh \to \Vh^{d+2},\quad \widetilde{L}_h[u_h] \coloneq (\epsilon_\theta u_h, \mathcal{P}_h^{\text{sp}} \gradsurf^h u_h, \mathcal{P}_h^{\text{temp}}K_h[u_h]).
\end{equation}
Since $\mathcal{P}_h^{\text{temp}}K_h[u_h]$ is piecewise affine in time, it suffices to evaluate it at the temporal nodes $t_s$, $s=0,\ldots, N_\T$, of $\mathcal{S}_h$. 
We therefore obtain
\begin{equation*}
\mathcal{P}_h^{\text{temp}}K_h[u_h](\x,t)= \sum_{i=1}^{N_\Vcal} \sum_{s,l=0}^{N_\T} \sum_{j=0}^{2 N_w} \mathbf{u}_{i,l} \mathbf{k}_j \D_{jl}(t_s) \varphi_i(\x)\rho_s(t)
\end{equation*}
where the entries of the temporal convolution matrix $\mathbf{D}(t_s)$ are given by
\begin{equation*}
\mathbf{D}(t_s)_{jl} = \int_{\widetilde{\T}} \rho_{j-N_w}(\tau) \mathbf{1}_\T(t_s + \tau)\rho_l(t_s + \tau) \dx \tau
= \int_{\widetilde{\T} }\rho_{j-N_w}(\tau) \mathbf{1}_\T(\tau + t_s)\rho_{l-s}(\tau) \dx\tau
\end{equation*}
On the uniform temporal grid, the nonzero entries are confined to the row index range
\begin{equation*}
    j_{\text{min}}(s)=\max\{0, s-N_w\}, \quad j_{\text{max}}(s)=\min\{2 N_w, N_w + N_\T-s\}. 
\end{equation*}
Within this range, $\mathbf{D}(t_s)$ has the structure of a shifted and boundary-truncated temporal mass matrix, with
\begin{align*}
\mathbf{D}(t_s)_{jl}
&=
\frac{\delta}{6} \left\{
\begin{array}{lll}
2, &\text{if } j-N_w = l-s, & \ j \in \{j_{\text{min}}(s), j_{\text{max}}(s)\} ,\\
4, &\text{if } j-N_w = l-s, & \ j_{\text{min}}(s) < j<  j_{\text{max}}(s),\\
1, &\text{if } |j-N_w -(l-s)| = 1, & \ j_{\text{min}}(s) \leq j  \leq j_{\text{max}}(s),\\
0, & \text{otherwise}.
\end{array}
\right.
\end{align*}
Here, the shift results from the equivalence of $t_s + \tau = t_l$ and $\tau=t_{l-s}$.
The computation of the discrete cross-correlation matrix is illustrated in~\cref{fig:discrete_cross_correlation}.
\begin{figure}[htbp]
\centering
\includegraphics[width=\linewidth]{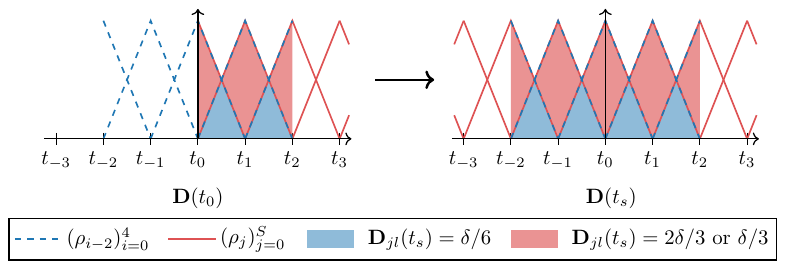}
\caption{Illustration of the discretized cross-correlation computation of the matrix $\D(t)$ evaluated at the nodes $t_s$ of the uniform mesh with kernel defined in $\widetilde{\T}=[t_{-2},t_2]$.
The integrals of the products of the overlapping hat functions give the matrix entries.}
\label{fig:discrete_cross_correlation}
\end{figure}

\subsection{Energy Functionals}
Let $\mathcal{J}_h(u_h, z_h)$ denote the finite element discretization of the continuous energy $\mathcal{J}(u,z)$ in~\eqref{eq:energy} such that we aim to optimize
\begin{equation}\label{eq:energy_discrete}
    \argmin_{u_h \in \Vh} 
    \left\{\mathcal{J}_h(u_h, z_h) \coloneqq G_h(u_h, z_h) + R_{\theta,h}(u_h)\right\},
\end{equation}
constructed from the discrete components $G_h$ with
\begin{equation*}
    G_h(u_h, z_h) = \frac{1}{2N_\Sigma}\sum_{i=1}^{N_\Sigma}\int_\T\left(\avint_{\mathcal{T}_h^{\Sigma_i}}A_h[u_h](\x,t) \dx S_h^\Gamma(\x)-z_{h,i}(t)\right)^2\dx t
\end{equation*}
and the regularizer function $R_{\theta,h}$ with potential functions $\widetilde{\phi_i}$ of~\eqref{eq:phi_tilde}
\begin{equation*}
    R_{\theta,h}(u_h) =\lambda_\theta\sum_{i=1}^{N_C}\int_{\mathcal{T}_h\times \T} \widetilde{\phi}_i(L_{h,i}[u_h](\x,t),\widetilde{L}_{h,i}[u_h](\x,t))\dxSh.
\end{equation*}
We extend the discrete functional in~\eqref{eq:energy_discrete} by setting it to $+\infty$ outside $\Vh$
\begin{equation*}
    \widetilde{\mathcal{J}}_h(u_h, z_h) =
\begin{cases}
\mathcal{J}_h(u_h, z_h), & \text{if } u_h \in \Vh,\\
+\infty, & \text{else}.
\end{cases}
\end{equation*}
For a fixed mesh size $h>0$, the discrete energies can be expressed as functions of the finite element degrees of freedom by evaluating with quadrature rules to compute the energy explicitly.
For the inverse problem, we define the linear operator notation $\widetilde{A}[u]_i(t) \coloneqq \avint_{\Sigma_i} A[u](\x,t)\dx S^\Gamma(\x)$ together with its matrix representation $\widetilde{\mathbf{A}}$, where the electrode averages are computed using the trapezoidal rule, which is exact for the piecewise affine finite element functions.
The discrete data fidelity term reads
\[
\mathbf{G}_h(\mathbf{u},\mathbf{z}) = \frac{1}{2N_\Sigma} (\widetilde{\mathbf{A}} \mathbf{u} - \mathbf{z})^\top 
\mathbf{D}^\Sigma (\widetilde{\mathbf{A}} \mathbf{u} - \mathbf{z}), 
\quad \text{with} \quad\mathbf{D}^\Sigma = \widetilde{\mathbf{D}} \otimes I_{N_\Sigma}.
\]
The discrete spatiotemporal regularizer is computed with the lumped quadrature rule
\[
\mathbf{R}_{\theta,h}(\mathbf{u}) = \lambda_\theta \sum_{i=1}^{N_C} \sum_{m=0}^{N_\T} 
\sum_{j=1}^{N_\mathcal{V}} \widetilde{\mathbf{D}}^{\text{lump}}_{mm} \widetilde{\mathbf{M}}^{\text{lump}}_{jj} 
\phi_i \left(
\begin{array}{c}
\epsilon_\theta \mathbf{u}_{j,m} \\
(\mathbf{P}^{\text{sp}}\pmb{\nabla}_{\Gamma_H}^h \mathbf{u})_{j,m} \\
(\mathbf{k}_i^\top \mathbf{D}(t_m) (\mathbf{u}_{j,s})_{s=0}^{N_\T})
\end{array}
\right),
\]
where $\widetilde{\mathbf{D}}^{\text{lump}}$ and $\widetilde{\mathbf{M}}^{\text{lump}}$ denote the lumped mass matrices.
Remark that we computed the regularizer with $\phi_i$ instead of $\widetilde{\phi}_i$ to omit unnecessary computation in practical applications.
This formulation provides a fully discrete representation of the energy functional in terms of the finite element degrees of freedom, facilitating efficient numerical implementation.
\begin{remark}\label{remark:quadrature_energies}
    The functional $\mathbf{G}_h$ is an exact representation of $G_h$ since the piecewise affine functions are evaluated with an exact quadrature rule.
    The joint regularizer $\mathbf{R}_{\theta,h}$ employs mass-lumped quadrature in space and time. 
    The resulting quadrature approximation is consistent under mesh refinement.
\end{remark}

\subsection{Geometric Approximation and Lifting}
Following~\cite{El13}, we define the lifting of functions between discretized and true domains.
We begin by introducing the bijective lifting from the discretized spatial domains to the continuous domain. 
The normal projection $\mathbf{x} \mapsto p(\mathbf{x})$ from a narrow band around $\partial \Omega_0$ is then defined as the unique solution of
$\x=p(\x)+d(\x)\mathbf{n}(p(\x))$
with the distance function
\begin{equation*}
d(\x)=
\begin{cases}
\begin{aligned}
&-\inf_{\mathbf{y}\in \partial\Omega_0}\left\{|\x-\mathbf{y}|\right\},  &&\mathbf{x}\in \Omega_0, \\
&\inf_{\mathbf{y}\in \partial\Omega_0}\left\{|\x-\mathbf{y}|\right\}, &&\x\notin \overline{\Omega}_0,
\end{aligned}
\end{cases}
\end{equation*}
and the normal to the boundary $\mathbf{n}(\x)=\nabla_\x d(\x)$ for $\x\in\partial\Omega_0$.
We assume that $h$ is chosen small enough such that all vertices of $\partial\mathcal{U}_h$ lie within a narrow band where the distance function and closest point projection are well defined.
Furthermore, the surface gradient of a function $v:\Gamma_H\to\mathbb{R}$ is defined as
$\gradsurf v = (\mathbf{I}_d - \mathbf{n} \otimes \mathbf{n})\nabla_\x v$
with $\nabla_\x v$ the gradient in the ambient coordinates of an arbitrary extension of $v$ to the narrow band.
The definition of the homeomorphism $H_h:\mathcal{U}_h\to\Omega_0$ follows~\cite{El13} by mapping each element of the computational domain back to a reference element and then to the true curved domain.
Denote the mapping from the reference element $\hat{J}$ to an element $J\in\mathcal{U}_h$ by $M_J$ and the mapping from the reference element to the curved element $J^e$ representing the true domain with curved boundary by $M_J^e$.
Then
\begin{equation*}
    H_h(\x)=M_J^e((M_J)^{-1}(\x)).
\end{equation*}
Restricted to the interior elements, the homeomorphism $H_h$ coincides with the identity.
\Cref{fig:lifting} illustrates the lifting of a boundary simplex in $\mathbb{R}^2$. 
A similar illustration for  $\mathbb{R}^3$ can be seen in~\cite{El13}.

\begin{figure}[htbp]
\centering
\includegraphics[width=\linewidth]{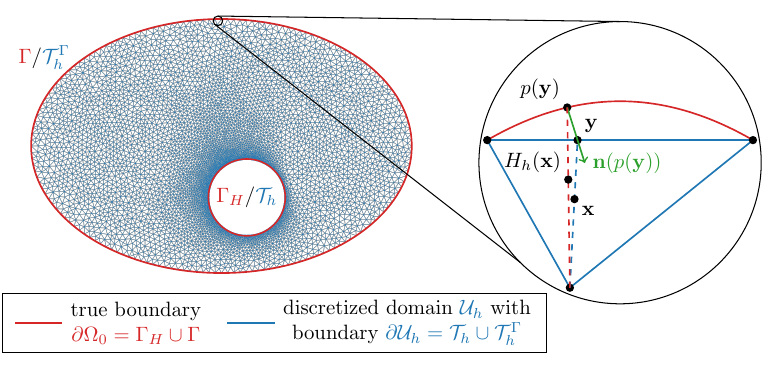}
\caption{Lifting of the $\x$ point in a simplex at a boundary element $J\subset \mathcal{U}_h$ to the true domain by $H_h(\x)\in J^e\subset \Omega_0$ in $\mathbb{R}^2$.
We compute the linear projection $\mathbf{y}$ of $\x$ onto the boundary $\mathcal{T}_h^\Gamma$, take the closest-point projection $p(\mathbf{y})$ on $\Gamma$, and map $\x$ towards $p(\mathbf{y})$ to compute $H_h(x)$.}
\label{fig:lifting}
\end{figure}
To describe the local behavior of the discrete boundary mapping near $\partial \Omega_0$, we use the following estimates.
For elements $J\in \mathcal{U}_h$ having a boundary face, the homeomorphism satisfies
\[
\|\nabla_\x H_h^\top|_J - \mathrm{Id}\|_{L^\infty(J)} \leq c h 
\quad \text{and} \quad \norm{\det \nabla_\x H_h^\top|_J| - 1}_{L^\infty(J)} \leq c h,
\]
where $c>0$ is a constant independent of $h$. 
Furthermore, for any discretization $\partial\mathcal{U}_h$ of the boundary $\partial\Omega_0$, we define the quotient of measures $\nu_h$ and $\nu_h^{\Gamma}$ such that $\dx S = \nu_h \dx S_h$ and $\dx S^\Gamma(\x) = \nu_h^\Gamma \dx S_h^\Gamma(\x)$ for surface measures $S_h$ and $S_h^\Gamma$ of $\mathcal{T}_h$ and $\mathcal{T}_h^\Gamma$, respectively.
The measures account for the geometric difference between the manifold and its discretizations, satisfying
\begin{equation}\label{eq:mani_meas_conv}
    \sup_{\mathcal{T}_h}\left\{|1-\nu_h|\right\}\leq ch \quad \text{and} \quad \sup_{\mathcal{T}_h^{\Gamma}}\left\{|1-\nu_h^{\Gamma}|\right\}\leq ch.
\end{equation}
As a result, we define a projection lifting a discretized signal to the continuous domain.
\begin{definition}[Lift and inverse lift]\label{def:lifting}
For a function $v_h:\mathcal{U}_h\to\mathbb{R}$ we define its \emph{lift} by
$v_h^l := v_h\circ H_h^{-1}$ on $\Omega_0$, and for $v:\Omega_0\to\mathbb{R}$ its \emph{inverse lift} by
$v^{-l} := v\circ H_h$ on $\mathcal{U}_h$.
Analogously, for $u_h:\partial\mathcal{U}_h\to\mathbb{R}$ the lift $u_h^l:\partial\Omega_0\to\mathbb{R}$ is defined by
$u_h^l(p(\mathbf{x})):=u_h(\mathbf{x})$, and for $u:\partial\Omega_0\to\mathbb{R}$ the inverse lift by
$u^{-l}(\mathbf{x}):=u(p(\mathbf{x}))$.
\end{definition}
As a consequence of these definitions, functions and their lifts are comparable in the relevant Sobolev norms, see~\cite[Proposition 4.9, 4.13]{El13}.
\begin{proposition}[Norm equivalence under lifting]\label{pro:norm_equiv}
Let $v_h:\mathcal{U}_h\to\mathbb{R}$ and $u_h:\partial\mathcal{U}_h\to\mathbb{R}$ with lifts
$v_h^l:\Omega_0\to\mathbb{R}$ and $u_h^l:\partial\Omega_0\to\mathbb{R}$.
Then the $L^2$- and $H^1$-norms of $v_h$ and $u_h$ are equivalent to the corresponding norms of their lifts.
The equivalence constants are independent of the discretization parameter $h$.
\end{proposition}

\section{Convergence}
In this section, we prove the Mosco convergence of the finite element discretized energy to the corresponding continuous energy for the convex subclass of the proposed regularization framework. Specifically, we consider potential functions $\phi_i$ given by the convex formulation in \eqref{eq:convex_potentials}. 
Therefore, the following analysis applies only to the convex subclass of the proposed model, for which required lower semicontinuity and coercivity properties can be established.
Moreover, we establish the convergence of minimizers.
The proof exploits some concepts of~\cite{Na24}.
We start by defining the topology for convergence.
\begin{definition}{($\mathcal{Z}$-topology)}
    Let $\Gamma_H$ be a surface with discretizations defined by $(\mathcal{T}_h)_h$, and let $(\mathcal{U}_h)_h$ be a family of shape-regular bulk discretizations of $\Omega_0$. 
    Let $u_h\in\Vh$ for every $h>0$. 
    We say that $(u_h)_h$  converges weakly to $u\in\Vcal$ w.r.t. the $\mathcal{Z}$-topology denoted by $u_h \xrightharpoonup{\mathcal{Z}} u$ if and only if the lifted functions satisfy $u_h^l \xrightharpoonup{\Vcal} u$.
    We say that  $(u_h)_h$ converges strongly to $u$ w.r.t. the $\mathcal{Z}$-topology $u_h \xrightarrow{\mathcal{Z}} u$ if and only if $u_h^l \xrightarrow{\mathcal{V}} u$.
\end{definition}
We recall the definition of Mosco convergence \cite{Mo69}.
\begin{definition}{(Mosco convergence)}
Functionals $\mathcal{J}_h:\Vh \to \overline{\mathbb{R}}$ for $h>0$ are said to converge to $\mathcal{J}:\Vcal \to \overline{\mathbb{R}}$ in the sense of Mosco w.r.t. the $\mathcal{Z}$-topology if
    \begin{enumerate}
    \item   
    for every sequence $(u_h)_h \subset \Vh$ with $u_h \xrightharpoonup{\mathcal{Z}} u $ the functional inequality $\mathcal{J}(u) \le \liminf_{h\to 0} \mathcal{J}_h(u_h)$ holds true ('\emph{liminf-inequality}').
    \item  
    for every $u\in\Vcal$ there exists a recovery sequence $(u_h)_h\subset \Vh$ such that  $u_h \xrightarrow{\mathcal{Z}} u$ and $\mathcal{J}(u) \ge \limsup_{h\to 0} \mathcal{J}_h(u_h)$ ('\emph{limsup-inequality}').
\end{enumerate}
If in $1.$ the strong convergence $u_h \xrightarrow{\mathcal{Z}} u$ is required,  then $\mathcal{J}_h$ is said to \(\Gamma\)-converge to $\mathcal{J}$ with respect to the $\mathcal{Z}$-topology.
\end{definition}
First, we prove the convergence of the discrete operator $\widetilde{L}_h$ defined in \eqref{eq:discrete_operator_2} to the continuous operator $L[u]=(\epsilon_\theta u, \gradsurf u, K[u])$.
\begin{lemma}\label{lem:operator convergence}
Let $(u_h)_h\subset\Vh$ be a sequence whose lifts $u_h^l$ converge weakly in $\Vcal$ to some $u\in\Vcal$.
Assume that the kernels $k_h$ converge strongly to $k$ in 
$L^1(\widetilde{\T})$.
Then the lifted operator \[
\widetilde{L}_h^l[u_h^l]
\coloneqq
\left(
\epsilon_\theta u_h^l,\,
(\mathcal{P}_h^{\mathrm{sp}}\gradsurf^h u_h)^l,\,
\mathcal{P}_h^{\mathrm{temp}}(k_h\ast_{\T}u_h^l)
\right)
\]
converges weakly in
$\Ycal^{d+2}=\Ycal\times\Ycal^d\times\Ycal$
to $L[u]$.
Moreover, if $u_h^l$ converges strongly to $u$ in $\Vcal$, then
$\widetilde{L}_h^l[u_h^l]$ converges strongly to $L[u]$ in
$\Ycal^{d+2}$.
\begin{proof}
Assume that $u_h^l \rightharpoonup u$ weakly in $\Vcal$. 
We prove the weak convergence of $\widetilde{L}_h^l[u_h^l]$ componentwise in $\Ycal^{d+2}$.  
For the first component, we have $\epsilon_\theta u_h^l \rightharpoonup \epsilon_\theta u \quad \text{in } \Ycal$,
which follows directly from the weak convergence of $u_h^l$ in $\Vcal$ and the continuous embedding $\Vcal \hookrightarrow \Ycal$.
For the spatial gradient component, let $\varphi \in \Ycal^d$ be an arbitrary test function and compute
\begin{equation*}
\int_{\Gamma_H\times \T} (\mathcal{P}_h^{\text{sp}} \gradsurf^h u_h)^l \cdot \varphi  \dxS = \int_{\mathcal{T}_h\times \T}  \gradsurf^h u_h \cdot \mathcal{P}_h^{\text{sp}} \varphi^{-l}  \dxSh + o(1),
\end{equation*}
where $\dxSh\coloneqq \dx S_h(\x)\otimes \dx\lambda(t)$  and the equality follows from self-adjointness of the $L^2$-projection $\mathcal{P}_h^{\text{sp}}$ on $\Yh^d$.  
We then decompose the integrand as
\begin{align*}
\int_{\mathcal{T}_h\times \T}  \gradsurf^h u_h \cdot \mathcal{P}_h^{\text{sp}} \varphi^{-l} \dxSh
&= \int_{\mathcal{T}_h\times \T}  \gradsurf^h u_h \cdot \varphi^{-l} \dxSh \\
& + \int_{\mathcal{T}_h\times \T}  \gradsurf^h u_h \cdot (\mathcal{P}_h^{\text{sp}} \varphi^{-l} - \varphi^{-l})   \dxSh.
\end{align*}
For the first term, the weak convergence of the lifted discrete gradients, $( \gradsurf^h u_h)^l \rightharpoonup  \gradsurf u$ in $\Ycal^d$, together with the change of variables from $\mathcal{T}_h$ to $\Gamma_H$, gives
\[
\int_{\mathcal{T}_h\times \T}  \gradsurf^h u_h \cdot \varphi^{-l}  \dxSh \to
\int_{\Gamma_H\times \T}  \gradsurf u \cdot \varphi  \dxS
\]
as $h\to 0$.
For the second term, we apply the Cauchy–Schwarz inequality together with the uniform boundedness of $\gradsurf^h u_h$ in $L^2(\mathcal{T}_h\times \T)$ such that
\begin{align*}
&\left|\int_{\mathcal{T}_h\times \T}  \gradsurf^h u_h \cdot (\mathcal{P}_h^{\text{sp}} \varphi^{-l} - \varphi^{-l})   \dxSh\right| \\
&\quad \leq C \, \min_{v_h^l \in (\Yh^l)^d} \left( \int_{\Gamma_H\times \T} |v_h^l - \varphi|^2  \dxS \right)^{1/2}.
\end{align*}
By the denseness in $\norm{\cdot}_{\Ycal^d}$, this last term tends to $0$ as $h \to 0$. 
Hence we conclude that $(\mathcal{P}_h^{\text{sp}} \gradsurf^h u_h)^l \rightharpoonup \gradsurf u \quad \text{in } \Ycal^d$.
For the temporal component, we first consider the continuous
cross-correlation operator $K[u]=k*_{\T}u$, which is compact, since the embedding
$H^1(\Gamma_H)\hookrightarrow L^2(\Gamma_H)$ is compact, and the
temporal convolution operator induced by $k\in L^1(\widetilde{\T})$
is compact on $L^2(\T)$. 
Hence, the weak convergence $u_h^l\rightharpoonup u$ in $\Vcal$
implies $k*_{\T}u_h^l\to k*_{\T}u$ strongly in $ \Ycal$.
Moreover, Young's inequality gives
\[
\begin{aligned}
\|(k_h-k)*_{\T}u_h^l\|_{\Ycal}
&\leq
\|k_h-k\|_{L^1(\widetilde{\T})}
\|u_h^l\|_{\Ycal}
\longrightarrow 0,
\end{aligned}
\]
since $(u_h^l)_h$ is bounded in $\Vcal$ and hence in $\Ycal$.
Therefore, $
k_h*_{\T}u_h^l\to k*_{\T}u$ strongly in $\Ycal$.
It remains to account for the temporal interpolation.
Set $v_h\coloneqq k_h*_{\T}u_h^l$ and $v\coloneqq k*_{\T}u$.
Choose a sequence $q_h$, piecewise affine in time, such that $q_h\to v$ strongly in $\Ycal$.
Since $\mathcal{P}_h^{\mathrm{temp}}q_h=q_h$ and
$\mathcal{P}_h^{\mathrm{temp}}$ is uniformly bounded in $L^2$ on the
piecewise cubic temporal finite element space, we obtain
\begin{align*}
\|\mathcal{P}_h^{\mathrm{temp}}v_h-v\|_{\Ycal}
&\leq
\|\mathcal{P}_h^{\mathrm{temp}}(v_h-q_h)\|_{\Ycal}
+\|q_h-v\|_{\Ycal}\\
&\leq
C\|v_h-q_h\|_{\Ycal}
+\|q_h-v\|_{\Ycal}\\
&\leq
C\|v_h-v\|_{\Ycal}
+(C+1)\|q_h-v\|_{\Ycal}
\longrightarrow0.
\end{align*}
Thus, $\mathcal{P}_h^{\mathrm{temp}}
(k_h*_{\T}u_h^l)
\to
k*_{\T}u$ strongly in $\Ycal$,
and in particular weakly in $\Ycal$.

Next, assume that $u_h^l \rightarrow u$ strongly in $\Vcal$. 
We again prove the strong convergence of $\widetilde{L}_h^l[u_h^l]$ componentwise in $\Ycal^{d+2}$.  
For the first component, strong convergence follows directly from the assumption $\epsilon_\theta u_h^l \rightarrow \epsilon_\theta u$ in $\Ycal$.
For the spatial gradient, we decompose the convergence error using an intermediate approximation $q_h \in \Yh^d$ which is a piecewise affine function whose lift satisfies $q_h^l \rightarrow \gradsurf u$ strongly in $\Ycal^d$. 
The existence of such a sequence is guaranteed by the approximation property of $(\Yh^l)^d$.  
We then write
\begin{align*}
\| (\mathcal{P}_h^{\text{sp}} \gradsurf^h u_h)^l - \gradsurf u \|_{\Ycal^d} 
&\le \| (\mathcal{P}_h^{\text{sp}} \gradsurf^h u_h)^l - (\mathcal{P}_h^{\text{sp}} q_h)^l \|_{\Ycal^d} \\
&+ \| (\mathcal{P}_h^{\text{sp}} q_h)^l - q_h^l \|_{\Ycal^d} 
+ \| q_h^l - \gradsurf u \|_{\Ycal^d}.
\end{align*}
For the first term, lifting back to $\mathcal{T}_h$ and using the stability of $\mathcal{P}_h^{\text{sp}}$ gives
\begin{align*}
&\| (\mathcal{P}_h^{\text{sp}} \gradsurf^h u_h)^l - (\mathcal{P}_h^{\text{sp}} q_h)^l \|_{\Ycal^d}^2
= \int_{\mathcal{T}_h \times \T} \left| \mathcal{P}_h^{\text{sp}} (\gradsurf^h u_h - q_h) \right|^2 \nu_h \, \dxSh \\
&\le C\int_{\mathcal{T}_h \times \T} \left| \gradsurf^h u_h - q_h \right|^2 \nu_h \dxSh \\
&\le C\left(\| (\gradsurf^h u_h)^l - \gradsurf u \|_{\Ycal^d}^2 + \| \gradsurf u - q_h^l\|_{\Ycal^d}^2\right),
\end{align*}
for $C>0$ induced by norm equivalence.
This converges to $0$ as $h \to 0$ by the strong convergence assumption and the choice of $q_h$.  
The second term vanishes exactly, since $\mathcal{P}_h^{\text{sp}} q_h = q_h$.  
The third term converges to $0$ by construction of $q_h$.  
Hence we conclude $(\mathcal{P}_h^{\text{sp}} \gradsurf^h u_h)^l \rightarrow \gradsurf u$ strongly in $\Ycal^d$.
Finally, for the temporal component, Young's inequality yields
\begin{align*}
\|k_h*_{\T}u_h^l-k*_{\T}u\|_{\Ycal}
&\leq
\|(k_h-k)*_{\T}u_h^l\|_{\Ycal}
+
\|k*_{\T}(u_h^l-u)\|_{\Ycal}\\
&\leq
\|k_h-k\|_{L^1(\widetilde{\T})}
\|u_h^l\|_{\Ycal}
+
\|k\|_{L^1(\widetilde{\T})}
\|u_h^l-u\|_{\Ycal}
\longrightarrow0.
\end{align*}
Set $v_h\coloneqq k_h*_{\T}u_h^l$ and $v\coloneqq k*_{\T}u$.
Choose a sequence $q_h$, piecewise affine in time, such that
$q_h\to v$ strongly in $\Ycal$. Since
$\mathcal{P}_h^{\mathrm{temp}}q_h=q_h$ and
$\mathcal{P}_h^{\mathrm{temp}}$ is uniformly bounded in $L^2$ on the
piecewise cubic temporal finite element space, we obtain
\begin{align*}
\|\mathcal{P}_h^{\mathrm{temp}}v_h-v\|_{\Ycal}
&\leq
\|\mathcal{P}_h^{\mathrm{temp}}(v_h-q_h)\|_{\Ycal}
+\|q_h-v\|_{\Ycal}\\
&\leq
C\|v_h-v\|_{\Ycal}
+(C+1)\|q_h-v\|_{\Ycal}
\longrightarrow0.
\end{align*}
Consequently, $\mathcal{P}_h^{\mathrm{temp}} (k_h*_{\T}u_h^l) \to k*_{\T}u$ strongly in $\Ycal$, which completes the proof.
\end{proof}
\end{lemma}
\begin{remark}
The weak and strong convergence results of \cref{lem:operator convergence} also hold for the operator $L_h$ in~\eqref{eq:discrete_operator_1} in place of $\widetilde{L}_h$.  
Both discretizations converge to the same limit operator $L$.
\end{remark}
Next, we consider the convergence of the solution of the forward problem.
\begin{lemma}\label{lem:forward_operator_conv}
Let $v\in L^2(\T;H^1(\Omega_0))$ be the weak solution of~\eqref{eq:E_forward} w.r.t. $u\in \Vcal$.
For each $h>0$, denote by $v_h\in P_1(\mathcal{U}_h)\otimes P_1(\mathcal{S}_h)$ the weak solution of the discretized forward problem~\eqref{eq:discrete_problem} with $v_h = u_h$ on $\mathcal{T}_h\times \T$.
Let $u_h^l$ and $v_h^l$ denote the lifted functions.
Then the following statements hold:
\begin{enumerate}
    \item If $u_h^l \rightharpoonup u$ weakly in $\Vcal$, then $v_h^l \rightharpoonup v$ weakly in $L^2(\T;H^1(\Omega_0))$.
    \item If $u_h^l \xrightarrow{} u$ strongly in $\Vcal$, then $v_h^l \xrightarrow{} v$ strongly in $L^2(\T;H^1(\Omega_0))$.
\end{enumerate}
\begin{proof}
By lifting the function, we can rewrite~\eqref{eq:discrete_problem} as finding $v_h^l$ with epicardium boundary function $u_h^l$ such that
\begin{equation*}
\begin{aligned}
a_h^l(v_h^l(\cdot,t),\varphi_h^l)
\coloneqq
\int_{\Omega_0}
&\sigma_h(H_h^{-1}(\x))
\nabla_\x H_h^\top(H_h^{-1}(\x))
\nabla_\x v_h^l(\x,t)
\\
&\cdot
\left(
\nabla_\x H_h^\top(H_h^{-1}(\x))
\nabla_\x \varphi_h^l(\x)
\right)
\frac{1}{
\left|
\det \nabla_\x H_h^\top(H_h^{-1}(\x))
\right|
}
\,\dx\x
=0.
\end{aligned}
\end{equation*}
for all $\varphi_h^l\in (P_1(\mathcal{U}_h))^l$ with $\varphi_h^l|_{\Gamma_H}=0$.
By the uniform ellipticity of $\sigma_h$ and the geometric estimates for $H_h$, the bilinear forms $a_h^l$ are uniformly continuous and coercive for sufficiently small $h$.
Then, we rewrite the weak problem as in~\eqref{eq:discrete_problem_2} using $v_h^l=w_h^l +g_h^l$, which yields
\begin{equation*}
    a_h^l(w_h^l(\cdot,t),\varphi_h^l)=-a_h^l(g_h^l(\cdot,t),\varphi_h^l) \quad \forall \varphi_h^l\in (P_1(\mathcal{U}_h))^l,
\qquad \varphi_h^l|_{\Gamma_H}=0.
\end{equation*}
Testing with $\varphi_h^l=w_h^l(\cdot,t)$ for almost all $t\in \T$ yields
\begin{equation*}
\norm{w^l_h(\cdot,t)}_{H^1(\Omega_0)}\leq C \norm{g^l_h(\cdot,t)}_{H^1(\Omega_0)}.
\end{equation*}
Hence, there exists an upper bound on the solution to the problem
\begin{equation*}
    \norm{v^l_h(\cdot,t)}_{H^1(\Omega_0)}\leq \norm{g^l_h(\cdot,t)}_{H^1(\Omega_0)} + \norm{w^l_h(\cdot,t)}_{H^1(\Omega_0)}\leq C \norm{u^l_h(\cdot,t)}_{H^{1/2}(\Gamma_H)}.
\end{equation*}
By integrating over $\T$, we prove an upper bound in $L^2(\T;H^1(\Omega_0))$ and by weak compactness up to a subsequence, we have $v_h^l\rightharpoonup \widetilde{v}$ weakly in $L^2(\T;H^1(\Omega_0))$.
Choose any $\varphi\in H^1(\Omega_0)$ with
$\varphi|_{\Gamma_H}=0$, any $\eta\in C_c^\infty(\T)$, and a sequence
$\varphi_h\in P_1(\mathcal{U}_h)$ such that $\varphi_h^l\to\varphi$ strongly in $H^1(\Omega_0)$.
Testing the discrete problem with $\eta(t)\varphi_h^l$ and integrating over
$\T$ yields
\begin{equation*}
0
=
\int_\T
a_h^l(v_h^l(\cdot,t),\varphi_h^l)\eta(t)\dx t
\longrightarrow
\int_\T
a(\widetilde v(\cdot,t),\varphi)\eta(t)\dx t.
\end{equation*}
Since $\eta\in C_c^\infty(\T)$ is arbitrary, it follows that $a(\widetilde v(\cdot,t),\varphi)=0$ for a.e. $t\in\T$.
Moreover, for almost every $t\in\T$, $\widetilde{v}(\cdot, t)|_{\Gamma_H}=u(\cdot, t)$  and by uniqueness of the continuous problem $\widetilde{v}(\cdot, t)=v(\cdot, t)$ which proves the weak convergence in $L^2(\T;H^1(\Omega_0))$.

Next, we assume that $u_h^l\xrightarrow[]{} u$ strongly in $\Vcal$.
For almost all $t\in T$, there exists an upper bound on the true solution to the problem
\begin{equation*}
    \norm{v(\cdot,t)-v^l_h(\cdot,t)}_{H^1(\Omega_0)}\leq \norm{g(\cdot,t)-g^l_h(\cdot,t)}_{H^1(\Omega_0)} + \norm{w(\cdot,t)-w^l_h(\cdot,t)}_{H^1(\Omega_0)}.
\end{equation*}
    The first part of the sum can be estimated from above by the trace extension operator on the epicardium, such that
    \begin{equation*}
        \norm{g(\cdot,t)-g^l_h(\cdot,t)}_{H^1(\Omega_0)}\leq C_E \norm{u(\cdot,t)-u_h^l(\cdot,t)}_{H^{1/2}(\Gamma_H)}.
    \end{equation*}
    Consequently,
\[
\|g-g_h^l\|_{L^2(\T;H^1(\Omega_0))}
\leq
C_E
\|u-u_h^l\|_{L^2(\T;H^{1/2}(\Gamma_H))}
\longrightarrow 0,
\]
where we used the continuous embedding
$H^1(\Gamma_H)\hookrightarrow H^{1/2}(\Gamma_H)$.
    For the upper bound of the second error, we refer the reader to~\cite[Theorem 6.1]{El13}.
    Integrating over time yields $v_h^l\to v$ strongly in $L^2(\T;H^1(\Omega_0))$.
\end{proof}
\end{lemma}
Next, we prove the convergence of the energy functionals.
\begin{lemma}\label{lem:convergence_regularizer_data} Let $R_\theta$ be convex, $z_h\xrightarrow{}z$ in $(L^2(T))^{N_\Sigma}$, and $k_{h,i}\xrightarrow{}k_i$ in $L^1(\T)$ for $i=1,\ldots,N_C$.
Then for any sequence $(u_h)_h\subset\Vh$ with $u_h\xrightharpoonup{\mathcal{Z}} u\in\Vcal$, we have $G(u,z) \leq \liminf_{h\to 0}G_h(u_h,z_h)$ and $ R_\theta(u) \leq \liminf_{h\to 0}R_{\theta,h}(u_h)$.
If $u_h \xrightarrow{\mathcal{Z}} u $, then $G(u,z) = \lim_{h\to 0}G_h(u_h,z_h)$ and $R_\theta(u) = \lim_{h\to 0}R_{\theta,h}(u_h)$.
\begin{proof}
Assume that $u_h\xrightharpoonup{\mathcal{Z}} u$ and therefore $u_h^l\xrightharpoonup{\Vcal} u$.
Denote by $v_h^l$ the lifted solution with respect to the boundary condition $u_h^l$ of the weak formulation~\eqref{eq:discrete_problem}.
By \cref{lem:forward_operator_conv}, we know that
$v_h^l \xrightharpoonup v$ in $L^2(\T;H^1(\Omega_0))$.
Define $A_h^l$ as the lifted operator depending on the lifted solution $v_h^l$ of the discretized problem~\eqref{eq:discrete_problem} for boundary condition $u_h^l$.
By concatenation with the linear bounded trace operator, weak convergence is preserved
$A_h^l[u_h^l] \xrightharpoonup A[u]$ in $L^2(\T;H^{1/2}(\Gamma))$.
The results follow from the lower semicontinuity of $D$ and the strong convergence of $z_h$.

Since the $\phi_i$ are convex and continuous in all arguments, $R_\theta$ is weak lower semicontinuous and we have together with~\Cref{lem:operator convergence}
\begin{align*}
     R_\theta(u)&\leq \liminf_{h\to 0} \lambda_\theta\sum_{i=1}^{N_C}\int_{\Gamma_H\times \T} \widetilde{\phi}_i(L_{h,i}^l[u_h^l](\x,t),\widetilde{L}_{h,i}^l[u_h^l](\x,t))\dxS \\
     &= \liminf_{h\to 0} \lambda_\theta\sum_{i=1}^{N_C}\int_{\mathcal{T}_h\times \T} \widetilde{\phi}_i(L_{h,i}[u_h](\x,t),\widetilde{L}_{h,i}[u_h](\x,t))\nu_h\dxSh.
\end{align*}
By~\eqref{eq:mani_meas_conv}, we can drop $\nu_h$ in the limit $h\to 0$ and conclude $R_\theta(u)\leq \liminf_{h\to 0} R_{\theta,h}(u_h)$.
Assume $u_h\xrightarrow{\mathcal{Z}} u$, then the lifted sequence converges $u_h^l\xrightarrow{\Vcal} u$,
\cref{lem:forward_operator_conv} together with the trace theorem yields, for $i=1,\ldots,N_\Sigma$,
\begin{equation*}
    \norm{A[u]-A_h^l[u_h^l]}_{L^2(\T;H^{1/2}(\Sigma_i))}\leq C_T\norm{v-v_h^l}_{L^2(\T;H^1(\Omega_0))}\xrightarrow{}0.
\end{equation*}
Moreover, by the Cauchy--Schwarz inequality,
\begin{equation*}
\left\|
\avint_{\Sigma_i}
\left(A[u]-A_h^l[u_h^l]\right)\dx S^\Gamma
\right\|_{L^2(\T)}^2
\leq
\frac{1}{|\Sigma_i|}
\|A[u]-A_h^l[u_h^l]\|_{L^2(\T;L^2(\Sigma_i))}^2
\xrightarrow{}0.
\end{equation*}
Therefore, together with the strong convergence of $z_h$
\begin{equation*}
\int_\T\left(\avint_{\Sigma_i}A[u]\dx S^\Gamma(\x) -z_i\right)^2\dx t  = \int_\T\left(\avint_{\Sigma_i}A_h^l[u_h^l]\dx S^\Gamma(\x) -z_{h,i}\right)^2\dx t + o(1).
\end{equation*}
By applying the inverse lifting to the boundary manifold subject to $A_h^l[u_h^l](p(\x),t)=A_h[u_h](\x,t)$, we can rewrite it such that 
\begin{equation*} 
\int_\T\left(\avint_{\Sigma_i}A_h^l[u_h^l]\dx S^\Gamma(\x) -z_{h,i}\right)^2\dx t=\int_\T\left(\avint_{\mathcal{T}^{\Sigma_i}}A_h[u_h]\dx S_h^\Gamma(\x) -z_{h,i}\right)^2\dx t + o(1).
\end{equation*}
Summing over all electrodes $i=1,\dots,N_\Sigma$ and scaling by $(2N_\Sigma)^{-1}$ yields the claim.

We show the convergence of the regularizer using Vitali's convergence theorem.
By~\Cref{lem:operator convergence} and the preceding remark,
$u_h^l \xrightarrow{\Vcal} u$ implies
\[
L_{h,i}^l[u_h^l]\xrightarrow{\Ycal^{d+2}}L_i[u]
\qquad\text{and}\qquad
\widetilde{L}_{h,i}^l[u_h^l]\xrightarrow{\Ycal^{d+2}}L_i[u].
\]
Hence, both sequences converge in measure on $\Gamma_H\times\T$.
By continuity of $\widetilde{\phi}_i$ in both arguments, it follows that
\[
\widetilde{\phi}_i
\left(
L_{h,i}^l[u_h^l],
\widetilde{L}_{h,i}^l[u_h^l]
\right)
\to
\phi_i(L_i[u])
\]
in measure on $\Gamma_H\times\T$.
For integrability, we show a growth estimate on $\widetilde{\phi}_i$.
We start with bounding $\widetilde{\omega}_\mu^{d+2}$ from above
$\| \mathrm{Proj}_{B_{\ell^1}}(\mathbf{y}/\mu) \|_2\leq \| \mathrm{Proj}_{B_{\ell^1}}(\mathbf{y}/\mu) \|_1 \leq 1$
and by the triangle inequality 
\begin{equation*}
\left\| \mathbf{y} - \mu\,\mathrm{Proj}_{B_{\ell^1}}\!\left(\mathbf{y}/\mu\right) \right\|_{\infty}\leq \norm{\mathbf{y}}_\infty + \mu \norm{\mathrm{Proj}_{B_{\ell^1}}\!\left(\mathbf{y}/\mu\right)}_\infty\leq \norm{\mathbf{y}}_2 + \mu.
\end{equation*}
Both combined ensure the boundedness $\widetilde{\omega}_\mu^{d+2}(\mathbf{y},\widetilde{\mathbf{y}})\leq C_1+C_2(\norm{\mathbf{y}}_2^2+\norm{\widetilde{\mathbf{y}}}_2^2)$ and by the boundedness of $\norm{\mathbf{Q}_i}$, we have the growth condition
\begin{equation*}
    |\widetilde{\phi}_i(\mathbf{y},\widetilde{\mathbf{y}})|= \mu_i |\widetilde{\omega}_{\mu_i}^{d+2} (\mathbf{y},\widetilde{\mathbf{y}})| - \mu_i|\widetilde{\omega}_{\eta_i\mu_i}^{d+2} (\mathbf{Q}_i\mathbf{y},\mathbf{Q}_i\widetilde{\mathbf{y}})|\leq\widetilde{C}_1 +\widetilde{C}_2(\norm{\mathbf{y}}_2^2+\norm{\widetilde{\mathbf{y}}}_2^2).
\end{equation*}
Since $L_{h,i}^l[u_h^l]\to L_i[u]$ and
$\widetilde{L}_{h,i}^l[u_h^l]\to L_i[u]$ converge strongly in
$\Ycal^{d+2}$, the families $
(|L_{h,i}^l[u_h^l]|^2)_h$ and $(|\widetilde{L}_{h,i}^l[u_h^l]|^2)_h$
are uniformly integrable in $L^1(\Gamma_H\times\T)$.
Together with the quadratic growth estimate, this implies that
$(\widetilde{\phi}_i(
L_{h,i}^l[u_h^l],\widetilde{L}_{h,i}^l[u_h^l]))_h$
is uniformly integrable in $L^1(\Gamma_H\times\T)$.
Hence, Vitali's convergence theorem yields
\begin{align*}
     R_\theta(u)&=\lim_{h\to 0} \lambda_\theta\sum_{i=1}^{N_C}\int_{\Gamma_H\times \T} \widetilde{\phi}_i(L_{h,i}^l[u_h^l](\x,t),\widetilde{L}_{h,i}^l[u_h^l](\x,t))\dxS \\
     &= \lim_{h\to 0} \lambda_\theta\sum_{i=1}^{N_C}\int_{\mathcal{T}_h\times \T} \widetilde{\phi}_i(L_{h,i}[u_h](\x,t),\widetilde{L}_{h,i}[u_h](\x,t))\nu_h\dxSh = \lim_{h\to 0} R_{\theta,h}(u_h).
\end{align*}
\end{proof}
\end{lemma}
Next, we prove Mosco convergence and convergence of minimizers.
\begin{theorem}{(Mosco convergence)}\label{th:gamma_convergence}
Let the assumptions of \Cref{lem:convergence_regularizer_data} be satisfied.
Then $\widetilde{\mathcal{J}}_h$ Mosco-converges to $\mathcal{J}$ with respect to the $\mathcal{Z}$-topology.
\begin{proof}
('\emph{liminf-inequality}')
Let $(u_h)_h \subset \mathcal{V}_h$ be a sequence such that $u_h \overset{\mathcal{Z}}{\rightharpoonup} u \in \mathcal{V}$. 
Thus, we have $u_h^l \overset{\Vcal}{\rightharpoonup} u$.
By~\cref{lem:convergence_regularizer_data}, it follows that
\[
R_\theta(u)
\leq
\liminf_{h\to 0} R_{\theta,h}(u_h)
\qquad\text{and}\qquad
G(u,z)
\leq
\liminf_{h\to 0} G_h(u_h,z_h).
\]
Combining these results, we obtain
\[
\mathcal{J}(u,z)
\leq
\liminf_{h\to 0}
\widetilde{\mathcal{J}}_h(u_h,z_h),
\]
which establishes the $\liminf$ condition.

('\emph{limsup-inequality}')
Since $\Gamma_H$ is a closed Riemannian manifold, $C^\infty(\Gamma_H)$ is dense in $H^1(\Gamma_H)$ as proven in~\cite{He96}.
Applying mollification in time to $u\in\mathcal{V}$, we obtain
$u\ast p_\eta\in C^\infty(\overline{\T};H^1(\Gamma_H))$ converging in
$\mathcal{V}$ as $\eta\to0$.
Hence, we can construct a smooth sequence
$(u_n)_n\subset C^\infty(\Gamma_H\times\overline{\T})$ such that
$u_n\xrightarrow{\mathcal{V}}u$.
For each $u_n$, let $u_{n,h}\in\Vh$ be its finite element interpolant.
By the standard interpolation estimates and the geometric estimates for
the lift, $u_{n,h}^l\xrightarrow{\Vcal}u_n$
 as $h\to0$.
By a diagonal argument, we can choose a sequence $u_h\in\Vh$ such that $u_h^l\xrightarrow{\Vcal}u$,
and hence $u_h\xrightarrow{\mathcal Z}u$.
Finally, applying~\cref{lem:convergence_regularizer_data} gives
\[
\lim_{h \to 0} \widetilde{\mathcal{J}}_h(u_h, z_h) = \lim_{h \to 0} G_h(u_h,z_h) + R_{\theta,h}(u_h)=G(u,z) + R_\theta(u,z)= \mathcal{J}(u,z),
\]
which establishes the limsup inequality and completes the proof.
\end{proof}
\end{theorem}
Finally, we establish convergence of the discrete minimizers to the continuous ones.
\begin{theorem}{(Convergence of minimizers)}\label{th:minimizer_convergence}
Let the assumptions of \Cref{lem:convergence_regularizer_data} be satisfied and $(u_h)_h \subset \mathcal{V}_h$ be a sequence of minimizers of the discrete energies $\widetilde{\mathcal{J}}_h$.
Then, the sequence $(u_h)_h$ converges weakly (up to a subsequence) in the $\mathcal{Z}$-topology to a minimizer of $\mathcal{J}$, and $\lim_{h \to 0} \inf_{u_h \in \mathcal{V}_h} \left\{\widetilde{\mathcal{J}}_h(u_h, z_h) \right\}
= \inf_{u \in \mathcal{V}} \left\{\mathcal{J}(u, z)\right\}$.
\begin{proof}
Since the sequence of minimizers is constructed in $\Vh$, we know that $\sup_{h>0} \{\widetilde{\mathcal{J}}_h(u_h, z_h)\} < \infty$.
By construction, the potential functions satisfy $\widetilde{\phi}_i(\mathbf{y},\widetilde{\mathbf{y}}) \geq \frac{\epsilon_\omega}{2}\|\mathbf{y}\|_2^2$, so for $C>0$ we have
\begin{align*}
    &R_{\theta,h}(u_h) 
    = \lambda_\theta \sum_{i=1}^{N_C} \int_{\mathcal{T}_h \times \T} \widetilde{\phi}_i(L_{h,i}[u_h](\x,t),\widetilde{L}_{h,i}[u_h](\x,t))\dxSh\\
    &\geq \lambda_\theta \frac{\epsilon_\omega}{2} \sum_{i=1}^{N_C} \int_{\mathcal{T}_h \times \T} 
      (\epsilon_\theta u_h(\x,t))^2 + |\gradsurf^h u_h(\x,t)|^2 + (K_{h,i}[u_h](\x,t))^2 \dxSh \\
    &\geq C \int_{\mathcal{T}_h \times \T} (u_h(\x,t))^2 + |\gradsurf^h u_h(\x,t)|^2 \dxSh
      = C \|u_h\|_{\mathcal{V}_h}^2\geq  \widetilde{C} \|u_h^l\|_{\Vcal}^2.
\end{align*}
By weak compactness, any bounded sequence in the Hilbert space $\Vcal$ has a 
weakly convergent subsequence. Hence, there exists a subsequence $u_h^l\overset{\Vcal}{\rightharpoonup} u^\infty$ and thus $u_h\overset{\mathcal{Z}}{\rightharpoonup} u^\infty$.
By the liminf-inequality in~\Cref{th:gamma_convergence}, we know that
\[
\mathcal{J}(u^\infty, z) \leq \liminf_{h \to 0} \widetilde{\mathcal{J}}_h(u_h, z_h) = \liminf_{h \to 0} \inf_{v_h \in \mathcal{V}_h}\left\{ \widetilde{\mathcal{J}}_h(v_h, z_h)\right\}.
\]
The corresponding upper bound follows directly from the limsup-inequality of~\Cref{th:gamma_convergence} by choosing a sequence $(\widetilde{u}_h)_h$ converging to a minimizer of $\mathcal{J}$
\[
\limsup_{h \to 0} \inf_{v_h \in \mathcal{V}_h} \left\{\widetilde{\mathcal{J}}_h(v_h, z_h) \right\}\leq \limsup_{h \to 0}  \widetilde{\mathcal{J}}_h(\widetilde{u}_h, z_h) = \inf_{v\in\Vcal} \left\{\mathcal{J}(v,z)\right\}.
\]
Finally, we conclude
\[
 \mathcal{J}(u^\infty, z)\leq \liminf_{h \to 0} \inf_{v_h \in \mathcal{V}_h} \left\{\widetilde{\mathcal{J}}_h(v_h, z_h)\right\} \leq \limsup_{h \to 0} \inf_{v_h \in \mathcal{V}_h} \widetilde{\mathcal{J}}_h(v_h, z_h) \leq \inf_{v\in\Vcal} \left\{\mathcal{J}(v,z)\right\}
\]
and therefore $u^\infty\in\argmin_{v\in\Vcal}\left\{\mathcal{J}(v,z)\right\}$.
\end{proof}
\end{theorem}
\begin{remark}
If $R_\theta$ is nonconvex, weak lower semicontinuity is generally not guaranteed, while strong convergence in $\mathcal V$ is sufficient to pass to the limit under suitable continuity and growth assumptions.
A natural approach to ensure this is to bound the minimizing sequence in a more regular space, for instance
\[
H^1(\T;H^{-1}(\Gamma_H)) \cap L^2(\T;H^2(\Gamma_H)),
\] 
whose norm could be added to the energy functional with a small weight~$(0<\epsilon \ll 1)$ to enforce additional regularity.
By Aubin--Lions lemma, this space embeds compactly into $\mathcal{V}$, thereby guaranteeing strong convergence of minimizing sequences in $\Vcal$.
\end{remark}

\section{Optimization}
This section presents the variational formulation and numerical minimization of the inverse problem and Gaussian denoising, together with the bi-level optimization approach used to learn the regularizer.

\subsection{Energy Minimization}
We train the multivariate FoE-type regularizer on Gaussian-denoised functions in order to learn a data-driven prior that captures the structural properties of epicardial potentials. 
The learned regularizer functional is expected to generalize to other inverse problems beyond pure denoising.
Let $v_h^m \in \Vh$, $m = 1,\dots,M$, denote noise-free inputs, scaled to the interval $[0,1]$. 
The corresponding noisy observations are defined by
\begin{equation*}
    z_h^m = v_h^m + \kappa^m n_h^m, \quad m = 1,\dots,M,
\end{equation*}
where $n_h^m \in \Vh$ is a noise realization whose nodal coefficient vector satisfies $\mathbf{n}^m \sim \mathcal{N}(0, \mathbf{I}_{N_\Vcal \times (N_\T+1)})$ and $\kappa^m > 0$ denotes the noise standard deviation.
The denoising problem associated with the regularizer $R_{\theta,h}$ is given by
\begin{equation}\label{eq:prox_map}
u_h^* \in \argmin_{u_h \in \Vh} \left\{\frac{1}{2}\|u_h - z_h\|_{\Yh}^2 + R_{\theta,h}(u_h)\right\}.
\end{equation}
In this formulation, Gaussian denoising corresponds to minimizing an energy functional of the same structure as in the inverse problem \eqref{eq:energy}, except that the original data fidelity term $G(u,z)$ is replaced by
$\widetilde{G}(u,z) \coloneqq \frac{1}{2}\|u - z\|_{\Ycal}^2$,
which measures the squared $L^2$-distance to the noisy input $z$.
Both the inverse problem \eqref{eq:energy} and the denoising problem \eqref{eq:prox_map} are minimized using an accelerated gradient descent method with restart, as proposed in~\cite{Du25} and summarized in~\Cref{alg:agd}.
We first compute the Fréchet derivatives of the functions and then discretize them, representing them as vector-valued quantities using nodal values and matrix forms, as summarized in~\cref{tab:agd_functions}.
For the regularizer function $R_\theta$, we restrict to only using the interpolated linear operators $\widetilde{L}_i$ for simplicity and define the linear operator $\widetilde{L}:u\mapsto (\widetilde{L}_1[u],\ldots, \widetilde{L}_{N_C}[u])$ and the vector valued derivative of the potential functions $\Phi(x_1,\ldots,x_{N_C}) = (\phi_1'(x_1),\ldots, \phi_{N_C}'(x_{N_C}))$ with
\begin{equation*}
    \phi_i'(\mathbf{y}) = \mu_i \Big(\mathrm{Proj}_{B_{\ell^1}}(\mathbf{y}/\mu_i) - \mathbf{Q}_i^\top \mathrm{Proj}_{B_{\ell^1}}(\mathbf{Q}_i \mathbf{y}/(\mu_i\eta_i))\Big) + \epsilon_\omega (\mathbf{I}-\mathbf{Q}_i^\top \mathbf{Q}_i/\eta_i)\mathbf{y},
\end{equation*}
where $\|\mathbf{Q}_i\|_\infty\leq 1$ and $\eta_i>\|\mathbf{Q}_i\|_2^2$ ensure that $\phi_i'$ is Lipschitz continuous with constant $(1+\epsilon_\omega)$. 
Furthermore, each linear operator $\widetilde{L}_i^*$ is discretized by
\begin{equation*}
    \widetilde{\mathbf{L}}_i^*\mathbf{q}= (\epsilon_\theta \mathbf{q}_1 + \sum_{j=1}^d \mathbf{M}^{-1} (\mathbf{P}^{\text{sp}} \pmb{\nabla}_{\Gamma_H,j})^\top \mathbf{M}\mathbf{q}_{j+1}+\mathbf{D}^{-1} (\mathbf{P}^{\text{temp}} \mathbf{K}_i)^\top \mathbf{D}\mathbf{q}_{d+2}
\end{equation*}
for vectors $\mathbf{q}=(\mathbf{q}_1,\ldots,\mathbf{q}_{d+2})\in \mathbb{R}^{N_\Vcal(N_\T+1)\times (d+2)}$,  to be adjoint in the finite element function spaces.
The spatial and temporal mass matrices appear only in the components where the operator acts in space or time, respectively; if a component depends solely on space (resp. time), the corresponding temporal (resp. spatial) mass matrix does not appear.
\begin{table}[htbp]
\renewcommand{\arraystretch}{1.3}
    \centering
    \begin{tabular}{|c|c|}
    \hline
     Fréchet derivative & 
     Discretization \\
        \hline
          $\mathcal{D}\widetilde{G}(u,z)(\x,t) = (u-z)(\x,t)$ & $\nabla \widetilde{\mathbf{G}}_h(\mathbf{u},\mathbf{z}) = \mathbf{u} - \mathbf{z}$ \\
          \hline
          $\mathcal{D}G(u,z)(\x,t) = \widetilde{A}^* \widetilde{A}[u](\x,t) - \widetilde{A}^*[z](\x,t)$ &
   $ \nabla \mathbf{G}_h(\mathbf{u},\mathbf{z}) = \mathbf{M}^{-1}\widetilde{\mathbf{A}}^\top (\widetilde{\mathbf{A}}\mathbf{u} - \mathbf{z})$\\
   \hline
   $\mathcal{D}R_\theta(u)(\x,t) = \lambda_\theta \widetilde{L}^* \, \Phi(\widetilde{L}[u])(\x,t)$ & $\nabla \mathbf{R}_{\theta,h}(\mathbf{u}) = \lambda_\theta \widetilde{\mathbf{L}}^* \Phi(\widetilde{\mathbf{L}} \mathbf{u})$ \\
   \hline
    \end{tabular}
    \caption{Fréchet derivatives and their discrete approximations employed in optimization.}
    \label{tab:agd_functions}
\end{table}

To determine the step size for~\Cref{alg:agd}, we compute the Lipschitz constants of the derivatives. We estimate the largest eigenvalue of $\widetilde{\mathbf{L}}^* \widetilde{\mathbf{L}}$ via the power method:
\begin{equation*}
    \mathbf{u}_{k+1} = \frac{\widetilde{\mathbf{L}}^* \widetilde{\mathbf{L}}\mathbf{u}_k}{\|\widetilde{\mathbf{L}}^* \widetilde{\mathbf{L}} \mathbf{u}_k\|_2}, \quad
    \lambda_{\max}(\widetilde{\mathbf{L}}^* \widetilde{\mathbf{L}}) \approx \|\widetilde{\mathbf{L}}^* \widetilde{\mathbf{L}} \mathbf{u}_{N_{\lambda_{\max}}}\|_2,
\end{equation*}
with maximum iteration $N_{\lambda_{\max}}\in\mathbb{N}$.
The Lipschitz constant of the derivative $\mathcal{D}(\widetilde{G}+R_\theta)$ determines the step sizes in the accelerated gradient descent algorithm and is computed by $\mathcal{L} = 1 + \lambda_\theta(1+\epsilon_\omega) \, \lambda_{\max}(\widetilde{\mathbf{L}}^* \widetilde{\mathbf{L}})$.
When computing reconstructions for the inverse problem, the Lipschitz constant $\mathcal{D}G$ is computed by $\lambda_{\max}(\widetilde{\mathbf{A}}^* \widetilde{\mathbf{A}})$.

\begin{algorithm}[tb]
\caption{Accelerated gradient descent with objective-based restart}\label{alg:agd}
\begin{enumerate}
    \item \textbf{Initialization:} Set $\mathbf{u}^0 = \mathbf{0.5}$, observation $\mathbf{z}$, Lipschitz constant $\mathcal{L}$ of $\mathcal{D}(\widetilde{G}+R_\theta)$, $\tau_0 = 1$, $f^0 = \infty$, $\mathbf{v}^0 = \mathbf{u}^0$.
    \item \textbf{Iteration} ($n \ge 0$): Compute
    \begin{align*}
    \begin{cases}
        \mathbf{u}^{n+1} = \mathbf{v}^n - \mathcal{L}^{-1} \nabla (\widetilde{\mathbf{G}}_h(\mathbf{v}^n,\mathbf{z}) + \mathbf{R}_{\theta,h}(\mathbf{v}^n)),\quad
        \tau_{n+1} = \frac{1 + \sqrt{1 + 4 \tau_n^2}}{2},\\
        \mathbf{v}^{n+1} = \mathbf{u}^{n+1} + \frac{\tau_n - 1}{\tau_{n+1}} (\mathbf{u}^{n+1} - \mathbf{u}^n),\quad
        f^{n+1} = \widetilde{\mathbf{G}}_h(\mathbf{v}^{n},\mathbf{z}) + \mathbf{R}_{\theta,h}(\mathbf{v}^{n})
        \end{cases}
    \end{align*}
    If $f^{n+1} > f^n$, reset: $\mathbf{v}^{n+1} = \mathbf{u}^{n+1}, \tau_{n+1} = 1$.
\end{enumerate}
\end{algorithm}

\subsection{Learning the Regularizer}
For regularizer generalization, we model $R_\theta$ as a function of the noise level $\kappa$ via the weighting parameters $\mu_i(\kappa)$ in~\eqref{eq:potential_function}. 
Further details are provided in~\cite{Du25}.
The regularizer is trained using the loss
\begin{equation*}
    \frac{1}{M}\sum_{m=1}^M \frac{1}{\sqrt{\kappa^m}}\|v_h^m - u_h^{*,m}\|_{\Yh},
\end{equation*}
for $u_h^{*,m}$ solving \eqref{eq:prox_map} for observations $z_h^m$. 
This results in a bi-level optimization problem: the inner problem optimizes $u_h\in\Vh$ by minimizing the energy, while the outer problem updates the parameters $\theta$ of the multivariate FoE regularizer~\eqref{eq:FoE_multivariate}. These parameters include the global weights $\lambda_\theta$ and $\epsilon_\theta$, the temporal kernel functions $k_i$, and the parameters of the potential functions $\phi_i$, which are represented by a neural network.
We compute the prior via denoising because the proximal operator of a regularizer effectively encodes the typical structure of noise-free functions: by learning $R_\theta$ such that $u_h^{*,m}$ removes Gaussian noise from measurements $z_h^m$, the regularizer captures the statistical properties of the clean functions $v_h^m$. 
Once trained, the learned prior can be transferred to other inverse problems, providing robust guidance in situations where direct inversion is unstable.

For the optimization of the neural network parameters, we implement the model in PyTorch~\cite{Pa19} and employ implicit differentiation to efficiently compute gradients through the equilibrium point defined by the model. 
Specifically, we use the \texttt{torchdeq} library~\cite{Ge23}, which provides scalable tools for deep equilibrium models, and adopt the Broyden algorithm as the fixed-point solver. 
This quasi-Newton method enables fast and memory-efficient convergence to the fixed point without requiring explicit backpropagation through all iterative solver steps. 
For optimizing the neural network parameters, we use the ADAM optimizer.
The training is performed on an NVIDIA A40 GPU.

\section{Results}
We train and benchmark learned regularizer approaches for denoising and the inverse problem on 2D simulations, comparing them with state-of-the-art handcrafted finite element regularizers, including spatiotemporal first-order Tikhonov and TV.

\subsection{Dataset Generation}
We simulate $1000$ epicardial potential fields on a 2D torso--heart model~\cite{Ga21} illustrated in~\Cref{fig:lifting}, which includes the lungs. 
For the 3D model displayed in~\Cref{fig:3D_model}, we simulate $200$ potentials on a torso--heart model generated according to the procedure described in \cite{Do26}, in which cardiac and torso anatomies, as well as the position and orientation of the heart within the torso, are varied using statistical shape models and subsequently discretized into three--dimensional finite element meshes.
The data are randomly divided into $80\%$ for training, $10\%$ for validation, and $10\%$ for testing for each of the models.
Synthetic data are generated by simulating cardiac activity on a finer heart mesh with a finite element reaction--diffusion model and interpolating to the coarser torso-heart model via nearest-neighbor interpolation to avoid inverse crimes. 
Additionally, the atrioventricular openings of the biventricular heart mesh are closed for the experiments, since we restrict the reconstruction to epicardial potentials.
The transmembrane potential $v : \Omega_H \times T \to \mathbb{R}$ is computed as proposed in~\cite{Po06}
\[
C_m \partial_t v - \frac{1}{\beta} \nabla \cdot (\mathbf{G}_m \nabla v) + I_{\mathrm{ion}}(v) = I_{\mathrm{stim}}, \quad 
I_{\mathrm{stim}}(\mathbf{x},t) =
\begin{cases}
I_{\max} \mathbf{1}_{\text{stim}}(\mathbf{x}), & 0 \le t < I_{\mathrm{dur}}, \\
0, & t \geq I_{\mathrm{dur}},
\end{cases}
\]
with the spatial indicator function for the randomly localised stimulus region $\mathbf{1}_{\text{stim}}$ in the 2D model, while in the 3D model, the stimulus is applied at one of three possible regions located at the left ventricle, right ventricle, or apex, see \Cref{fig:3D_model}.
\begin{figure}[htbp]
\centering
\includegraphics[width=0.7\linewidth]{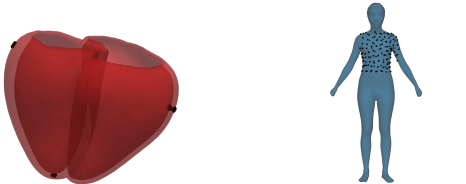}
\caption{Three-dimensional torso-heart model with three stimulation sites on the epicardium (apex, left and right ventricle) and human torso with $128$ electrodes.}
\label{fig:3D_model}
\end{figure}
Element-wise anisotropic intra- and extracellular conductivities $\mathbf{G}_{i}$ and $\mathbf{G}_{e}$
\[
\mathbf{G}_{\alpha} = \sigma_{\alpha,t} \mathbf{I} + (\sigma_{\alpha,l} - \sigma_{\alpha,t}) \mathbf{l} \otimes \mathbf{l}, \quad \alpha \in \{i,e\}, \quad \mathbf{G}_m = \mathbf{G}_i (\mathbf{G}_i + \mathbf{G}_e)^{-1} \mathbf{G}_e
\]
are computed, where $\mathbf{l}$ is the local fiber direction, which in the heart model is taken to be circumferential around the ventricles, reflecting the typical orientation of myocardial fibers. 
The conductivities are computed following~\cite{Ro02} by fixing $\sigma_{i,l}$ and determining the remaining components as functions of $\lambda_{\mathrm{LT}},\alpha$, and $\varepsilon$:
\begin{equation*}
    \sigma_{i,t}=\frac{\sigma_{i,l}}{\lambda_{\mathrm{LT}}^2}\left(\frac{1+\alpha(1-\varepsilon)}{1+\alpha}\right), \quad \sigma_{e,t}= \frac{\sigma_{i,t}}{\alpha(1-\varepsilon)}, \quad \sigma_{e,l}= \frac{\sigma_{i,l}}{\alpha}.
\end{equation*}
With probability $1/3$, a randomly localized region of scar tissue with a prescribed radius is introduced by reducing the conductivity tensor $\mathbf{G}_m$ by a factor drawn uniformly from $[0.05,0.25]$ and with probability $1/6$, a second scar region is added.
Ionic currents follow a Nagumo-type cubic model without repolarization,
\[
I_{\mathrm{ion}}(v) = g_{\max} (v - V_{\mathrm{rest}})(v - V_{\mathrm{th}})(v - V_{\mathrm{dep}}),
\]
with fixed depolarization, resting, and threshold potentials. 
Time integration uses backward Euler with $\Delta t$, yielding
\[
\left(C_m \mathbf{M} + \frac{\Delta t}{\beta} \mathbf{K}_m\right) \mathbf{v}^{n+1} = \mathbf{M} \left(C_m \mathbf{v}^n - \Delta t (I_{\mathrm{ion}}\left(\mathbf{v}^n\right) - I_{\mathrm{stim}})\right),
\]
where $\mathbf{K}_m$ corresponds to the finite element discretization of 
$-\nabla \cdot (G_m \nabla\,\cdot)$ and $\mathbf{v}$ contains the nodal values.
The extracellular potential $v_e$ is computed via the pseudo bidomain model~\cite{Bi11} for each $n_\mathrm{sample}$-th timestep 
\[
(\mathbf{K}_i + \mathbf{K}_e + \eta \mathbf{M}) \mathbf{v}_e = - \mathbf{K}_i \mathbf{v}\quad \text{with} \quad \eta = 10^{-9}.
\]
\Cref{fig:2D_gen_data} illustrates simulated epicardial potentials that are part of each dataset, including scar tissue. 
The simulation parameters used to generate these datasets are summarized in \Cref{tab:dataset_params}.
\begin{figure}[htbp]
\centering
\includegraphics[width=\linewidth]{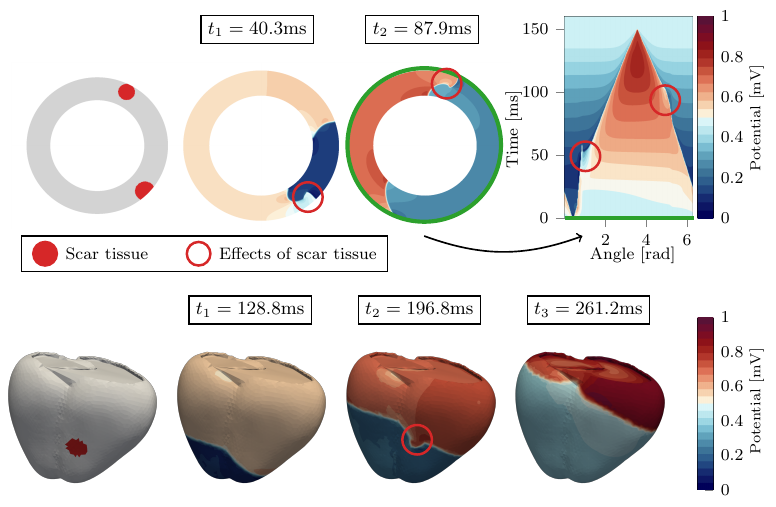}
\caption{Extracellular potential $v_e$ on the myocardium
for the 2D and 3D dataset at multiple time steps
with scar tissue and normalized to $[0,1]$.
The 2D data is visualized in a spacetime plot on the epicardium.
3D potentials are visualized restricted to the epicardium.
Reduced conductivity in the scar region deforms the characteristic
potentials and slows propagation.}
\label{fig:2D_gen_data}
\end{figure}
\begin{table}[htbp]
    \centering
    \begin{footnotesize}
    \begin{tabular}{|c|c|c|c|}
    \hline \textbf{Parameter} & \textbf{Value} & \textbf{Parameter} & \textbf{Value} \\
        \hline
        $V_{\mathrm{rest}}$
        & $-85\,\mathrm{mV}$
        & $V_{\mathrm{dep}}$
        & $30\,\mathrm{mV}$ \\
        \hline

        $V_{\mathrm{th}}$
        & $-55\,\mathrm{mV}$
        & $g_{\max}$
        & $\substack{1.4\times10^{-3}\ \text{(2D)}\\
        3.132\times10^{-4}
        \text{(3D)}}$ \\
        \hline

        $\sigma_{i,l}$
        & $3\,\mathrm{S/m}$
        & $\lambda_{\mathrm{LT}}$~\cite{Ro02}
        & $[2.16,2.84]$ \\
        \hline

        $\varepsilon$~\cite{Ro02}
        & $[0.58,0.93]$
        & $\alpha$~\cite{Ro02}
        & $1$ \\
        \hline

        $C_m$~\cite{Po06}
        & $1\,\mu\mathrm{F}/\mathrm{cm}^2$
        & $\beta$~\cite{Po06}
        & $\substack{800\,\mathrm{cm}^{-1}\ \text{(2D)}\\
        179\mathrm{cm}^{-1}\ \text{(3D)}}$ \\
        \hline

        $I_{\mathrm{dur}}$
        & $1.2\,\mu\mathrm{A}/\mathrm{cm}^2$
        & $I_{\max}$
        & $100\,\mathrm{ms}$ \\
        \hline

        $\Delta t$
        & $[0.08,0.088]\,\mathrm{ms}$
        & $n_{\mathrm{sample}}$
        & $\substack{[15,17]\ \text{(2D)}\\
                     [43,47]\ \text{(3D)}}$ \\
        \hline
    \end{tabular}
    \end{footnotesize}
    \caption{Model parameters and sampling ranges used for the generation
    of the synthetic datasets. Parameters are identical for the 2D and 3D
    models unless specified separately. Parameters are drawn uniformly at
    random from the indicated intervals.}
    \label{tab:dataset_params}
\end{table}

\subsection{Baseline Methods}
As handcrafted baselines, we consider spatiotemporal first-order Tikhonov (\textbf{TIK}) regularization,
\[
    \widetilde{G}(u,z) + \frac{1}{2}\norm{\Lambda\nabla_{(\x,t)}u}_{\Ycal^d}^2,
\]
where \(\Lambda = \mathrm{diag}(\lambda_{\gamma},\ldots,\lambda_{\gamma},\lambda_t)\in\R^{(d+1)\times(d+1)}_+\), and \(\lambda_{\gamma}\) and \(\lambda_t\) denote the spatial and temporal regularization parameters, respectively.  
The optimization problem is solved by applying the conjugate gradient method to the optimality system.
In addition, we compare against isotropic spatiotemporal total variation (\textbf{TV}) regularization \cite{Ha25},
\[
    \widetilde{G}(u,z) + \int_{\Gamma_H\times\T}\norm{\Lambda\nabla_{(\x,t)} u(\x,t)}_2 \dxS,
\]
which is minimized using a first-order primal--dual algorithm.
All baseline methods are computed using lumped mass matrices, as this significantly reduces computation time while producing visually and $L^2$-error-wise negligible differences.

\subsection{Denoising}
We first evaluate the FoE models on Gaussian denoising of epicardial potentials, where the ground truth (\textbf{GT}) functions are corrupted with Gaussian noise of standard deviation $\kappa$, matching the noise used during training. This section highlights the model's ability to reconstruct clean signals from noisy observations, motivating their use as priors for more general inverse problems.
Finally, we evaluate the proposed spatiotemporal FoE–type regularizer in both a convex (\textbf{CMFoE}) and a non-convex (\textbf{MFoE}) formulation. 
The convex version ensures theoretical guarantees and predictable behavior, while the non-convex version offers greater flexibility to capture complex signal correlations and potentially improve empirical performance.
The parameters used for tuning the models are the input noise level $\kappa$ and the learned weighting parameter $\lambda_\theta$.
To increase the range of the filters, we compute each filter as a composition of three filters defined on a smaller interval, increasing the number of output channels for each cross-correlation, i.e. $k_i(t) = \int_{\mathbb{R}^2} k_i^3(s)k_i^2(\tau)k_i^1(\tau+s+t)\dx (s,\tau)$ for $k_i^j\in L^1(\widetilde{\T})$ and $k_i^j=0$ on $\mathbb{R}\setminus \widetilde{\T}$.
The piecewise affine kernel functions $k_{h_i}$ learned from Gaussian denoising, with each $k^j_{h,i}$ defined on $\widetilde{\T}=[t_{-2},t_2]$ are shown in~\Cref{fig:kernels}.
The models are trained for $8000$ iterations with a learning rate of $0.001$ and a batch size of $1$ due to the varying number of temporal values of the nodal vectors.

\begin{figure}[htbp]
\centering
\includegraphics[width=\linewidth]{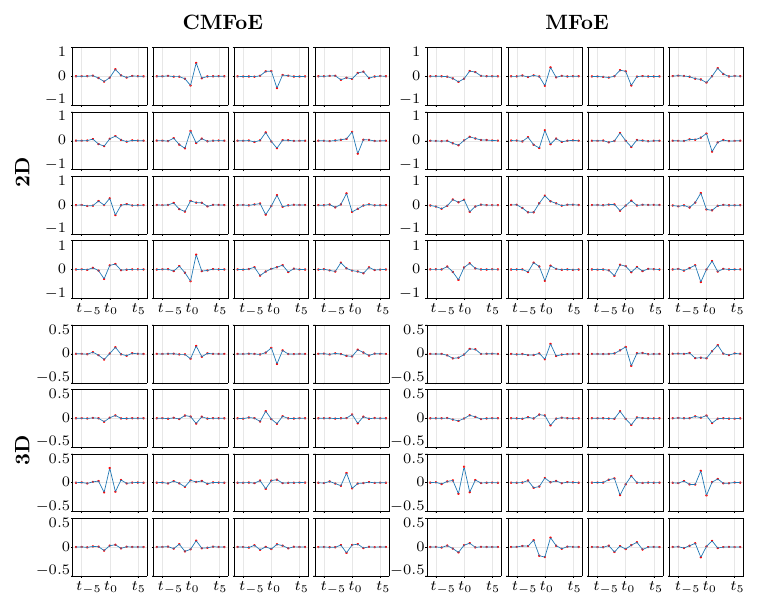}
\caption{Learned kernel functions $(k_{h,i})_{i=1}^{N_C}$ for $N_C=16$ of the FoE approaches by denoising.}
\label{fig:kernels}
\end{figure}

\Cref{fig:2D_model_denoising} and \Cref{fig:3D_model_denoising} illustrate representative denoising results for the 2D and 3D models, respectively, obtained with the different regularization approaches for noise level $\kappa=0.2$.
In~\Cref{tab:table_denoising}, errors for denoising across multiple different noise levels on the test set are computed by tuning the parameters on the validation set first.
As expected, $\mathbf{TIK}$ regularization produces overly smooth reconstructions: while it effectively reduces noise, it also blurs sharp features and fails to preserve edges, resulting in a loss of fine structural details.
In contrast, $\mathbf{TV}$ regularization yields reconstructions with significantly sharper edges. 
Discontinuities and piecewise constant regions are well preserved.
The spatiotemporal $\mathbf{CMFoE}$ and $\mathbf{MFoE}$ regularizers yield the best denoising performance with significant improvement for the nonconvex case, effectively suppressing Gaussian noise while preserving sharp spatial and temporal features.
Its learned filters adapt better to the noise characteristics than handcrafted regularizers, producing reconstructions that are visually and quantitatively closer to the ground truth.

\begin{figure}[htbp]
\centering
\includegraphics[width=\linewidth]{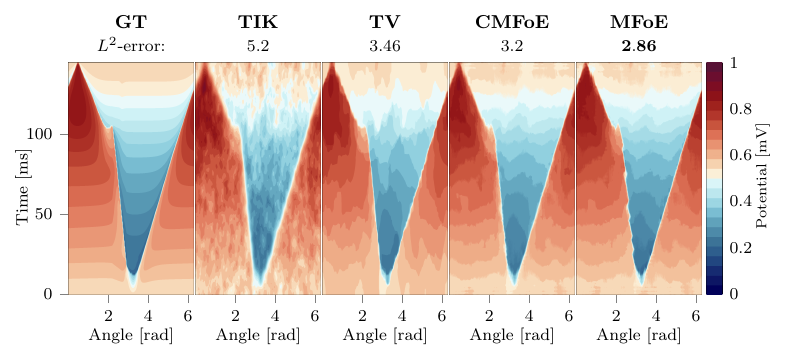}
\caption{Denoising reconstructions of the 2D model and $L^2$-errors of multiple regularization approaches for an observation with a noise level of $\kappa=0.2$ applied to the ground truth, visualized in a space-time plot.}
\label{fig:2D_model_denoising}
\end{figure}

\begin{figure}[htbp]
\centering
\includegraphics[width=\linewidth]{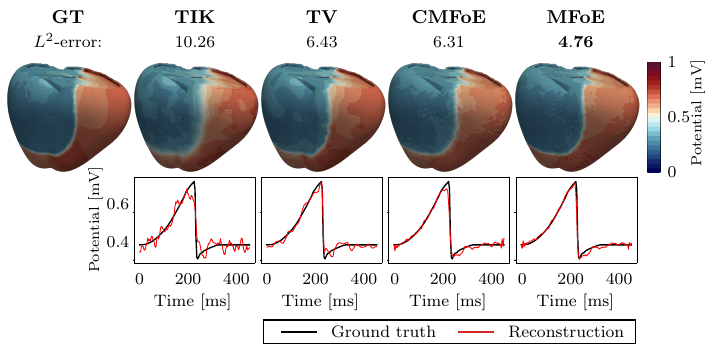}
\caption{Denoising reconstructions of the 3D model and $L^2$-errors of multiple regularization approaches for an observation with a noise level of $\kappa=0.2$ applied to the ground truth, visualized on the epicardium for one timestep and as a plot over time at a randomly chosen point on the epicardium.}
\label{fig:3D_model_denoising}
\end{figure}

\begin{table}[htbp]
    \centering
    \begin{footnotesize}
    \begin{tabular}{|c|c|c|c|c|}
    \hline
     \diagbox[width=4.7em]{$\kappa$}{Reg.} & \textbf{TIK} & \textbf{TV} & \textbf{CMFoE} & \textbf{MFoE} \\
        \hline
            & \multicolumn{4}{|c|}{\textbf{2D}} \\
        \hline 
          \multirow{2}*{$0.05$}&  $2.88$ & $1.92$ & $1.63$ & $\mathbf{1.36}$ \\
          \cline{2-5}
          & $\lambda_\gamma,\lambda_t=0.807,1.234$ & $\lambda_\gamma,\lambda_t=0.049,0.059$ & $\lambda_\theta,\kappa= 557,0.02$ & $\lambda_\theta,\kappa=489,0.05$\rule{0pt}{1em}\\
        \hline 
          \multirow{2}*{$0.1$}&  $4.08$ & $2.77$ & $2.33$ & $\mathbf{1.96}$ \\
        \cline{2-5}
          & $\lambda_\gamma,\lambda_t=1.364,2.084$ & $\lambda_\gamma,\lambda_t=0.09,0.101$ & $\lambda_\theta,\kappa= 557,0.1$ & $\lambda_\theta,\kappa=489,0.1$\rule{0pt}{1em}\\
        \hline 
          \multirow{2}*{$0.2$}&  $5.58$ & $3.81$ & $3.42$ & $\mathbf{2.85}$ \\
        \cline{2-5}
        & $\lambda_\gamma,\lambda_t=2.283,3.2$ & $\lambda_\gamma,\lambda_t=0.182,0.154$ & $\lambda_\theta,\kappa= 557,0.2$ & $\lambda_\theta,\kappa=489,0.2$\rule{0pt}{1em}\\
        \hline 
                & \multicolumn{4}{|c|}{\textbf{3D}} \\
        \hline 
          \multirow{2}*{$0.2$}&  $9.63$ & $6.41$ & $6.38$ & $\mathbf{4.92}$ \\
        \cline{2-5}
        & $\lambda_\gamma,\lambda_t=0.404,6.057$ & $\lambda_\gamma,\lambda_t=0.214,0.118$ & $\lambda_\theta,\kappa= 3244,0.2$ & $\lambda_\theta,\kappa=2227,0.2$\rule{0pt}{1em}\\
        \hline 
    \end{tabular}
    \end{footnotesize}
    \caption{Mean $L^2$-errors and corresponding regularization parameters for denoising reconstructions of different noise levels with standard deviation $\kappa\in(0.05,0.1,0.2)$ for 2D and $\kappa=0.2$ for 3D.}
    \label{tab:table_denoising}
\end{table}

\subsection{Inverse Problem}
We next consider the inverse problem in electrocardiographic imaging, again comparing $\mathbf{TIK}$ regularization, $\mathbf{TV}$ regularization, and the proposed spatiotemporal $\mathbf{CMFoE}$ and $\mathbf{MFoE}$ approaches as illustrated for the 2D model in~\cref{fig:2D_model_inverse} and for the 3D model in~\cref{fig:3D_model_inverse} and evaluated for different noise levels measured in signal-to-noise ratio (SNR) of the observations in~\cref{tab:2D_table_inverse}.
In contrast to the pure denoising setting, the inverse problem is severely ill-posed, and the reconstructions require substantially longer computation times due to the repeated solution of forward and adjoint problems within the optimization procedure.
The ill-posedness of the inverse problem is clearly reflected in the reconstruction quality. 
$\mathbf{TV}$ regularization improves edge preservation and produces sharper spatial features, but may introduce artificial discontinuities, especially in regions with low signal-to-noise ratio.
The $\mathbf{MFoE}$  regularizer again provides the best overall performance. 
Despite the increased difficulty of the inverse problem, it is able to capture relevant spatiotemporal activation patterns more accurately and to suppress noise and inversion artifacts more effectively than the classical approaches. 
In particular, the learned temporal interactions help stabilize the reconstruction and reduce non-physical oscillations over time.
Importantly, although the inverse problem is more challenging than in the denoising case, the ranking of the regularizers in terms of performance remains the same.

\begin{figure}[htbp]
\centering
\includegraphics[width=\linewidth]{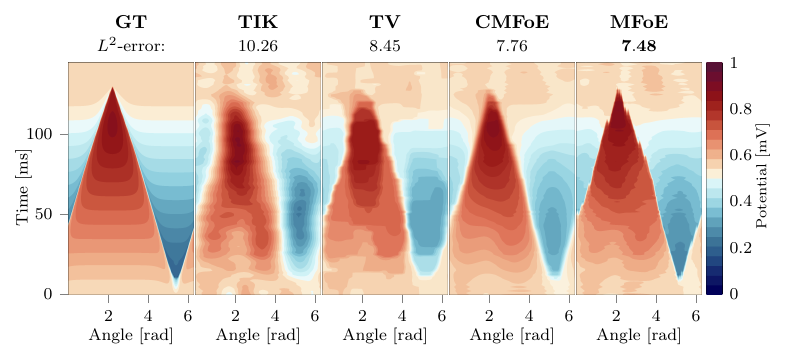}
\caption{Inverse problem reconstructions of the 2D model and $L^2$-errors of multiple regularization approaches with SNR of $30$ dB by applying Gaussian noise to the observations, visualized in a space-time plot.}
\label{fig:2D_model_inverse}
\end{figure}

\begin{figure}[htbp]
\centering
\includegraphics[width=\linewidth]{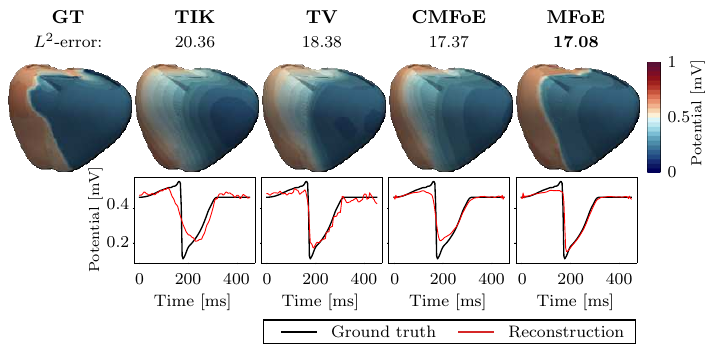}
\caption{Inverse problem reconstructions of the 3D model and $L^2$-errors of multiple regularization approaches with an SNR of $30$ dB by applying Gaussian noise to the observations, visualized on the epicardium for one time step and as a plot over time at a randomly chosen point on the epicardium.}
\label{fig:3D_model_inverse}
\end{figure}

\begin{table}[htbp]
    \centering
    \begin{footnotesize}
    \begin{tabular}{|c|c|c|c|c|}
    \hline
     \diagbox[width=4.7em]{SNR}{Reg.} & \textbf{TIK} & \textbf{TV} & \textbf{CMFoE} & \textbf{MFoE} \\
    \hline
    & \multicolumn{4}{|c|}{\textbf{2D}} \\
    \hline
        \multirow{3}{*}{$50$}
        & $8.73$ & $6.29$ & $5.94$ & $\mathbf{5.91}$ \\
        \cline{2-5}
        & $\lambda_\gamma  = 3.55 \times 10^{-3}$ 
        & $\lambda_\gamma = 6.57 \times 10^{-7}$ 
        & $\lambda_\theta = 4.14 \times 10^{-3}$ 
        & $\lambda_\theta = 1.03 \times 10^{-3}$ \rule{0pt}{1em}\\
        & $\lambda_t = 5.11 \times 10^{-3}$ 
        & $\lambda_t = 3.57 \times 10^{-7}$ 
        & $\kappa = 1.36 \times 10^{-1}$ 
        & $\kappa = 3.11 \times 10^{-1}$ \\
        \hline 
        \multirow{3}{*}{$40$}
        & $9.56$ & $7.46$ & $6.81$ & $\mathbf{6.66}$ \\
        \cline{2-5}
        & $\lambda_\gamma = 9.09 \times 10^{-3}$ 
        & $\lambda_\gamma = 4.28 \times 10^{-6}$ 
        & $\lambda_\theta = 2.06 \times 10^{-2}$ 
        & $\lambda_\theta = 8.42 \times 10^{-3}$ \rule{0pt}{1em}\\
        & $\lambda_t = 1.57 \times 10^{-2}$ 
        & $\lambda_t = 2.12 \times 10^{-6}$ 
        & $\kappa = 1.57 \times 10^{-1}$ 
        & $\kappa = 2.83 \times 10^{-1}$ \\
        \hline 
        \multirow{3}{*}{$30$}
        & $10.91$ & $9.26$ & $8.5$ & $\mathbf{8.34}$ \\
        \cline{2-5}
        & $\lambda_\gamma = 2.78 \times 10^{-2}$ 
        & $\lambda_\gamma = 2.32 \times 10^{-5}$ 
        & $\lambda_\theta = 0.252 \times 10^{-1}$ 
        & $\lambda_\theta = 3.32 \times 10^{-2}$ \rule{0pt}{1em}\\
        & $\lambda_t = 4.28 \times 10^{-2}$ 
        & $\lambda_t = 1.15 \times 10^{-5}$ 
        & $\kappa = 1.64 \times 10^{-1}$ 
        & $\kappa = 3.39 \times 10^{-1}$ \\
        \hline 
            & \multicolumn{4}{|c|}{\textbf{3D}} \\
    \hline
        \multirow{3}{*}{$30$}
        & $19.52$ & $17.33$ & $16.75$ & $\mathbf{16.58}$ \\
        \cline{2-5}
        & $\lambda_\gamma = 2.37 \times 10^{-3}$ 
        & $\lambda_\gamma = 7.27 \times 10^{-7}$ 
        & $\lambda_\theta = 4.15 \times 10^{-1}$ 
        & $\lambda_\theta = 2.6 \times 10^{-1}$ \rule{0pt}{1em}\\
        & $\lambda_t = 4.69 \times 10^{-2}$ 
        & $\lambda_t = 2.23 \times 10^{-6}$ 
        & $\kappa = 9.51 \times 10^{-2}$ 
        & $\kappa = 9.51 \times 10^{-2}$ \\
        \hline 
    \end{tabular}
    \end{footnotesize}
    \caption{Mean $L^2$-errors and corresponding regularization parameters for inverse problem reconstructions at SNR $(30,40,50)$ dB for 2D and SNR $30$ dB for 3D.}
    \label{tab:2D_table_inverse}
\end{table}

\section{Conclusion}
We introduced a class of trained spatiotemporal regularizers for inverse problems on unstructured meshes and analyzed them within a rigorous variational framework. 
For the convex subclass of the proposed regularization framework, the finite element discretized functionals are shown to Mosco-converge to their continuous counterparts, guaranteeing convergence of minimizers and stability under mesh refinement, while the more expressive nonconvex formulation is not covered by this theoretical analysis.
Numerical experiments on denoising and the inverse problem show that the proposed spatiotemporal regularizer consistently outperforms classical handcrafted approaches.
Although the inverse problem is severely ill-posed and reconstructions require longer computation times, the relative performance of the regularizers in contrast to each other remains unchanged.
The current implementation is limited to uniform timesteps, which restricts its applicability in problems requiring adaptive or variable temporal resolution. 
Extending the method to variable timesteps would increase computational complexity.
Similarly, extending the model to learn kernel convolutions on spatial manifolds, analogous to temporal convolutions, is computationally demanding, thereby limiting the practicality of spatial regularization.
Overall, this work shows that data-driven regularization can be combined with strong analytical guarantees, yielding methods that are both theoretically sound and practically superior for challenging inverse problems such as ECGI.

\section*{Acknowledgements}
We thank Arsenii Dokuchaev for generating the three-dimensional heart–torso geometry used in this work.

%\section*{Conflict of interest}
%The authors declare that they have no conflict of interest.

% \bibliographystyle{siamplain}
% \bibliography{literature.bib}

\begin{thebibliography}{10}

\bibitem{Be07}
{\sc F.~B. Belgacem}, {\em Why is the {C}auchy problem severely ill-posed?}, Inverse Problems, 23 (2007), pp.~823--836, \url{https://doi.org/10.1088/0266-5611/23/2/020}, \url{https://doi.org/10.1088/0266-5611/23/2/020}.

\bibitem{Bi11}
{\sc M.~J. Bishop and G.~Plank}, {\em Bidomain {ECG} {Simulations} {Using} an {Augmented} {Monodomain} {Model} for the {Cardiac} {Source}}, IEEE transactions on bio-medical engineering, 58 (2011), p.~10.1109/TBME.2011.2148718, \url{https://doi.org/10.1109/TBME.2011.2148718}, \url{https://www.ncbi.nlm.nih.gov/pmc/articles/PMC3378475/} (accessed 2022-07-19).

\bibitem{Br08}
{\sc S.~C. Brenner and L.~R. Scott}, {\em The mathematical theory of finite element methods}, vol.~15 of Texts in Applied Mathematics, Springer, New York, third~ed., 2008, \url{https://doi.org/10.1007/978-0-387-75934-0}, \url{https://doi.org/10.1007/978-0-387-75934-0}.

\bibitem{Ci78}
{\sc P.~G. Ciarlet}, {\em The finite element method for elliptic problems}, vol.~Vol. 4 of Studies in Mathematics and its Applications, North-Holland Publishing Co., Amsterdam-New York-Oxford, 1978.

\bibitem{Cl18}
{\sc M.~Cluitmans, D.~H. Brooks, R.~MacLeod, O.~Dössel, M.~S. Guillem, P.~M. van Dam, J.~Svehlikova, B.~He, J.~Sapp, L.~Wang, and L.~Bear}, {\em Validation and opportunities of electrocardiographic imaging: From technical achievements to clinical applications}, Frontiers in Physiology, 9 (2018), \url{https://doi.org/10.3389/fphys.2018.01305}, \url{https://www.frontiersin.org/articles/10.3389/fphys.2018.01305}.

\bibitem{Fr14}
{\sc P.~Colli~Franzone, L.~F. Pavarino, and S.~Scacchi}, {\em Mathematical cardiac electrophysiology}, vol.~13 of MS\&A. Modeling, Simulation and Applications, Springer, Cham, 2014, \url{https://doi.org/10.1007/978-3-319-04801-7}, \url{https://doi.org/10.1007/978-3-319-04801-7}.

\bibitem{Do26}
{\sc A.~Dokuchaev, F.~Bonizzoni, S.~Pagani, F.~Regazzoni, and S.~Pezzuto}, {\em Learning geometry-dependent lead-field operators for forward ecg modeling}, arXiv preprint arXiv:2602.22367,  (2026).

\bibitem{Du25}
{\sc S.~Ducotterd and M.~Unser}, {\em Multivariate fields of experts}, 2025, \url{https://arxiv.org/abs/2508.06490}, \url{https://arxiv.org/abs/2508.06490}.

\bibitem{Ef20}
{\sc A.~Effland, S.~Neumayer, and M.~Rumpf}, {\em Convergence of the time discrete metamorphosis model on {H}adamard manifolds}, SIAM J. Imaging Sci., 13 (2020), pp.~557--588, \url{https://doi.org/10.1137/19M1247073}, \url{https://doi.org/10.1137/19M1247073}.

\bibitem{El13}
{\sc C.~M. Elliott and T.~Ranner}, {\em Finite element analysis for a coupled bulk-surface partial differential equation}, IMA J. Numer. Anal., 33 (2013), pp.~377--402, \url{https://doi.org/10.1093/imanum/drs022}, \url{https://doi.org/10.1093/imanum/drs022}.

\bibitem{Ev10}
{\sc L.~C. Evans}, {\em Partial differential equations}, vol.~19 of Graduate Studies in Mathematics, American Mathematical Society, Providence, RI, 1998, \url{https://doi.org/10.1090/gsm/019}, \url{https://doi.org/10.1090/gsm/019}.

\bibitem{Ga21}
{\sc L.~Gander, R.~Krause, M.~Multerer, and S.~Pezzuto}, {\em Space-time shape uncertainty in the forward and inverse problem of electrocardiography}, International Journal for Numerical Methods in Biomedical Engineering, 37 (2021), \url{https://doi.org/10.1002/cnm.3522}, \url{https://arxiv.org/abs/2010.16104}.

\bibitem{Ge23}
{\sc Z.~Geng and J.~Z. Kolter}, {\em Torchdeq: A library for deep equilibrium models}.
\newblock \url{https://github.com/locuslab/torchdeq}, 2023.

\bibitem{Go24}
{\sc A.~Goujon, S.~Neumayer, and M.~Unser}, {\em Learning weakly convex regularizers for convergent image-reconstruction algorithms}, SIAM J. Imaging Sci., 17 (2024), pp.~91--115, \url{https://doi.org/10.1137/23M1565243}, \url{https://doi.org/10.1137/23M1565243}.

\bibitem{Gr24}
{\sc T.~Grandits, J.~Verh\"ulsdonk, G.~Haase, A.~Effland, and S.~Pezzuto}, {\em Digital twinning of cardiac electrophysiology models from the surface {ECG}: a geodesic backpropagation approach}, IEEE Transactions on Biomedical Engineering, 71 (2024), pp.~1281--1288, \url{https://doi.org/10.1109/TBME.2023.3331876}, \url{https://arxiv.org/abs/2308.08410}.

\bibitem{Ha25}
{\sc M.~Haas, T.~Grandits, T.~Pinetz, T.~Beiert, S.~Pezzuto, and A.~Effland}, {\em Finite {E}lement-{B}ased {S}pace-{T}ime {T}otal-{V}ariation-{T}ype {R}egularization of the {I}nverse {P}roblem in {E}lectrocardiographic {I}maging}, SIAM J. Sci. Comput., 47 (2025), pp.~B1317--B1342, \url{https://doi.org/10.1137/24M1685055}, \url{https://doi.org/10.1137/24M1685055}.

\bibitem{He96}
{\sc E.~Hebey}, {\em Sobolev spaces on {R}iemannian manifolds}, vol.~1635 of Lecture Notes in Mathematics, Springer-Verlag, Berlin, 1996, \url{https://doi.org/10.1007/BFb0092907}, \url{https://doi.org/10.1007/BFb0092907}.

\bibitem{Ka26}
{\sc M.~Kappes, M.~Haas, T.~Beiert, S.~Pezzuto, and A.~Effland}, {\em Total generalized variation for inverse problems involving piecewise linear finite elements}, 2026, \url{https://arxiv.org/abs/2609.12778}, \url{https://arxiv.org/abs/2609.12778}.

\bibitem{Ka18}
{\sc A.~Karoui, L.~Bear, P.~Migerditichan, and N.~Zemzemi}, {\em Evaluation of fifteen algorithms for the resolution of the electrocardiography imaging inverse problem using ex-vivo and in-silico data}, Frontiers in Physiology, 9 (2018), \url{https://doi.org/10.3389/fphys.2018.01708}.

\bibitem{Ko22}
{\sc E.~Kobler, A.~Effland, K.~Kunisch, and T.~Pock}, {\em Total deep variation: A stable regularization method for inverse problems}, IEEE Transactions on Pattern Analysis and Machine Intelligence, 44 (2022), pp.~9163--9180, \url{https://doi.org/10.1109/TPAMI.2021.3124086}.

\bibitem{Mo65}
{\sc J.~Moreau}, {\em Proximité et dualité dans un espace hilbertien}, Bulletin de la Société Mathématique de France, 93 (1965), pp.~273--299, \url{http://eudml.org/doc/87067}.

\bibitem{Mo69}
{\sc U.~Mosco}, {\em Convergence of convex sets and of solutions of variational inequalities}, Advances in Math., 3 (1969), pp.~510--585, \url{https://doi.org/10.1016/0001-8708(69)90009-7}, \url{https://doi.org/10.1016/0001-8708(69)90009-7}.

\bibitem{Na24}
{\sc G.~Nardi, B.~Charlier, and A.~Trouv\'e}, {\em The matching problem between functional shapes via a {$BV$} penalty term: a {$\Gamma$}-convergence result}, Interfaces Free Bound., 26 (2024), pp.~381--414, \url{https://doi.org/10.4171/ifb/517}, \url{https://doi.org/10.4171/ifb/517}.

\bibitem{On09}
{\sc M.~Onal and Y.~Serinagaoglu}, {\em Spatio-temporal solutions in inverse electrocardiography}, in 4th European Conference of the International Federation for Medical and Biological Engineering, J.~Vander~Sloten, P.~Verdonck, M.~Nyssen, and J.~Haueisen, eds., Berlin, Heidelberg, 2009, Springer Berlin Heidelberg, pp.~180--183.

\bibitem{Pa19}
{\sc A.~Paszke, S.~Gross, F.~Massa, A.~Lerer, J.~Bradbury, G.~Chanan, T.~Killeen, Z.~Lin, N.~Gimelshein, L.~Antiga, A.~Desmaison, A.~K\"{o}pf, E.~Yang, Z.~DeVito, M.~Raison, A.~Tejani, S.~Chilamkurthy, B.~Steiner, L.~Fang, J.~Bai, and S.~Chintala}, {\em PyTorch: an imperative style, high-performance deep learning library}, Curran Associates Inc., Red Hook, NY, USA, 2019.

\bibitem{Pe22}
{\sc S.~Pezzuto, P.~Perdikaris, and F.~Sahli~Costabal}, {\em Learning cardiac activation maps from 12-lead {ECG} with multi-fidelity {B}ayesian optimization on manifolds}, IFAC-PapersOnLine, 55 (2022), pp.~175--180, \url{https://doi.org/10.1016/j.ifacol.2022.09.091}, \url{https://arxiv.org/abs/2203.06222}.

\bibitem{Po06}
{\sc M.~Potse, B.~Dube, J.~Richer, A.~Vinet, and R.~M. Gulrajani}, {\em A comparison of monodomain and bidomain reaction-diffusion models for action potential propagation in the human heart}, IEEE Transactions on Biomedical Engineering, 53 (2006), pp.~2425--2435, \url{https://doi.org/10.1109/TBME.2006.880875}.

\bibitem{Ro02}
{\sc B.~J. Roth}, {\em Electrical conductivity values used with the bidomain model of cardiac tissue}, IEEE Transactions on Biomedical Engineering, 44 (2002), pp.~326--328.

\bibitem{Ro05}
{\sc S.~Roth and M.~J. Black}, {\em Fields of experts: a framework for learning image priors}, 2005 IEEE Computer Society Conference on Computer Vision and Pattern Recognition (CVPR'05), 2 (2005), pp.~860--867 vol. 2, \url{https://api.semanticscholar.org/CorpusID:2843211}.

\bibitem{Ru92}
{\sc L.~I. Rudin, S.~Osher, and E.~Fatemi}, {\em Nonlinear total variation based noise removal algorithms}, Physica D: Nonlinear Phenomena, 60 (1992), pp.~259--268, \url{https://doi.org/https://doi.org/10.1016/0167-2789(92)90242-F}, \url{https://www.sciencedirect.com/science/article/pii/016727899290242F}.

\bibitem{Se05}
{\sc M.~Seger, G.~Fischer, R.~Modre-Osprian, B.~Messnarz, F.~Hanser, and B.~Tilg}, {\em Lead field computation for the electrocardiographic inverse problem - finite elements versus boundary elements}, Computer methods and programs in biomedicine, 77 (2005), pp.~241--52, \url{https://doi.org/10.1016/j.cmpb.2004.10.005}.

\bibitem{Te22}
{\sc R.~Tenderini, S.~Pagani, A.~Quarteroni, and S.~Deparis}, {\em P{DE}-aware deep learning for inverse problems in cardiac electrophysiology}, SIAM J. Sci. Comput., 44 (2022), pp.~B605--B639, \url{https://doi.org/10.1137/21M1438529}, \url{https://doi.org/10.1137/21M1438529}.

\bibitem{Ti77}
{\sc A.~N. Tikhonov and V.~Y. Arsenin}, {\em Solutions of ill-posed problems}, Scripta Series in Mathematics, V. H. Winston \& Sons, Washington, DC; John Wiley \& Sons, New York-Toronto-London, 1977.
\newblock Translated from the Russian, Preface by translation editor Fritz John.

\bibitem{Un26}
{\sc M.~Unser and S.~Ducotterd}, {\em Universal architectures for the learning of polyhedral norms and convex regularizers}, SIAM Journal on Imaging Sciences, 19 (2026), pp.~1272--1299.

\bibitem{Wa10}
{\sc D.~Wang, R.~M. Kirby, and C.~R. Johnson}, {\em Resolution strategies for the finite-element-based solution of the ecg inverse problem}, IEEE Transactions on Biomedical Engineering, 57 (2010), pp.~220--237, \url{https://doi.org/10.1109/TBME.2009.2024928}.

\end{thebibliography}

\end{document}